\documentclass{IEEEtran}

\usepackage{cite}
\usepackage{amsmath,amssymb,amsfonts}
\usepackage{algorithmic}
\usepackage{graphicx}
\usepackage{textcomp}
\usepackage{amsthm}
\usepackage[linewidth=1pt]{mdframed}
\usepackage{booktabs} 
\usepackage{tabularx} 
\usepackage{tabularray}
\usepackage{mathtools}
\usepackage{hyperref}
\usepackage{dsfont}
\usepackage{cleveref}
\usepackage{siunitx}
\usepackage{tcolorbox}
\usepackage{tikz}
\usetikzlibrary{shapes,arrows,fit,positioning}

\newtheorem{theorem}{Theorem}[section]
\newtheorem{lemma}[theorem]{Lemma}
\newtheorem{proposition}[theorem]{Proposition}

\newtheorem{corollary}[theorem]{Corollary}

\newtheorem{assumption}{Assumption}

\newcounter{problemcounter}
\crefname{problemcounter}{ProblemCounter}{ProblemCounters}

\def\BibTeX{{\rm B\kern-.05em{\sc i\kern-.025em b}\kern-.08em
    T\kern-.1667em\lower.7ex\hbox{E}\kern-.125emX}}
\begin{document}
\title{Maximizing Reach-Avoid Probabilities for Linear Stochastic Systems via Control Architectures}
\author{Niklas Schmid$^{1}$, Jaeyoun Choi$^{2}$, Oswin So$^{2}$, and Chuchu Fan$^{2}$
    \thanks{$^{1}$Automatic Control Laboratory (IfA), ETH Z\"urich, 8092 Z\"urich, Switzerland
            {\tt\small nikschmid@ethz.ch}}%
    \thanks{$^{2}$Department of Aeronautics and Astronautics, MIT, Cambridge, Massachusetts, USA
            {\tt\small \{cjy0051, oswinso, chuchu\}@mit.edu}%
            }%
    \thanks{Work supported by the European Research Council under grant 787845 (OCAL) and the Swiss National Science Foundation under NCCR Automation under grant 51NF40\_225155. Niklas Schmid’s research visit at MIT was additionally supported by the Mobility Grant of the Swiss National Science Foundation under NCCR Automation.}%
    }

\maketitle

\begin{abstract}
    The maximization of reach-avoid probabilities for stochastic systems is a central topic in the control literature. Yet, the available methods are either restricted to low-dimensional systems or suffer from conservative approximations. To address these limitations, we propose control architectures that combine the flexibility of Markov Decision Processes with the scalability of Model Predictive Controllers. The Model Predictive Controller tracks reference signals while remaining agnostic to the stochasticity and reach-avoid objective. Instead, the reach-avoid probability is maximized by optimally updating the controller's reference online. To achieve this, the closed-loop system, consisting of the system and Model Predictive Controller, is abstracted as a Markov Decision Process in which a new reference can be chosen at every time-step. A feedback policy generating optimal references is then computed via Dynamic Programming. If the state space of the system is continuous, the Dynamic Programming algorithm must be executed on a finite system approximation. Modifications to the Model Predictive Controller enable a computationally efficient robustification of the Dynamic Programming algorithm to approximation errors, preserving bounds on the achieved reach-avoid probability. The approach is validated on a perturbed 12D quadcopter model in cluttered reach-avoid environments proving its flexibility and scalability.
\end{abstract}



\section{Introduction}
Formal certificates are of ever rising importance in the face of automation of safety-critical systems, such as UAVs \cite{schmid2022real}, air traffic \cite{prandini2008application} and resource management \cite{pitchford2007uncertainty}. For systems whose dynamics evolve deterministically, safety is popularly imposed via (constrained) optimal control approaches, e.g., Dynamic Programming (DP) \cite{lygeros1999controllers}, Reinforcement Learning (RL) \cite{garcia2015comprehensive}, Barrier Functions \cite{ames2016control}, Model Predictive Control (MPC) \cite{borrelli2017predictive}, and path planners \cite{karaman2011sampling}. While these methods usually offer trade-offs in terms of flexibility (of systems and control objectives) and scalability (in the state dimensionality or signal complexity), combining their strengths via cascaded control architectures is widely explored in robotics, control and learning \cite{dong2023review, matni2024towards,stamouli2025layered, fan2020fast, feher2020hierarchical, stulp2012reinforcement, benders2025embedded}. In these layered designs, flexible path planning algorithms, such as DP, RL or random trees, are used to guide scalable tracking algorithms such as MPC to perform complex control tasks. 

\begin{figure}[!ht]    
    \centering
    \includegraphics[width=.85\linewidth, trim={2.5cm 0cm 2.5cm 1cm}, clip]{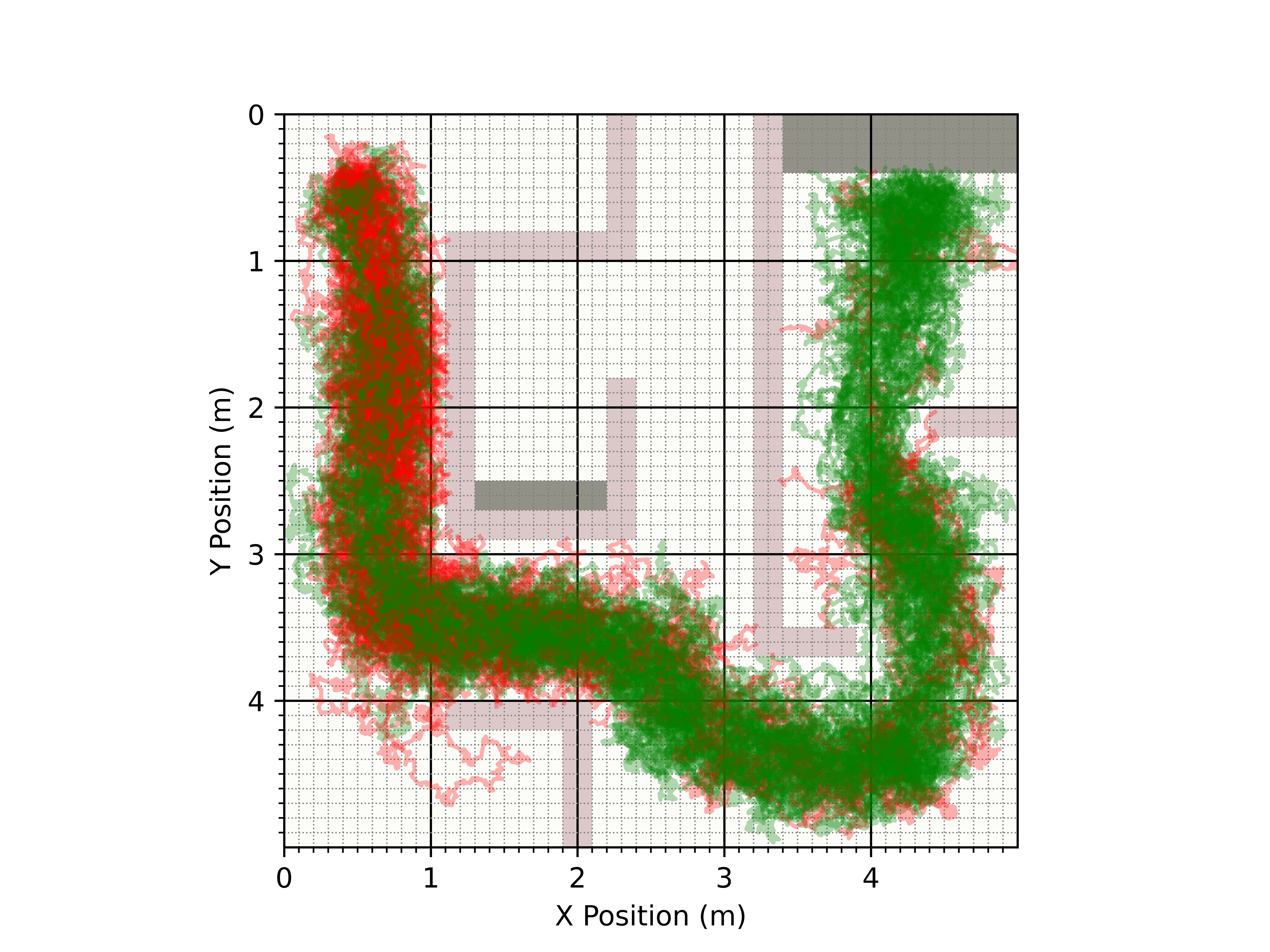}
    \caption{X-Y-trajectory of a perturbed quadcopter. The goal is to reach the gray sets while avoiding the red walls of the labyrinth. A control architecture is used to maximize the reach-avoid probability, achieving a success rate of $\SI{40}{\percent}$. Successful trajectories are green, unsuccessful red. The initial position is $(-0.5,-0.5)$. Details are provided in Section \ref{sec_numerical_example}.}
    \label{fig_num_lab}
\end{figure}

For stochastic systems, issues of scalability, or restrictions on the permissible system class, are even more pronounced. Generally, control approaches for safety-constrained stochastic systems are separable into two categories: 

The first category of controllers uses statistical bounds, which are based on moments or samples drawn from a distribution, to compute high-confidence uncertainty sets for the state trajectory. Robust control techniques are then used to ensure that the system trajectory is safe under all uncertainty realizations in this set, which consequently ensures a high probability of safety for the stochastic system. Respective approaches include stochastic tube \cite{ono_jcc_mpc,bavdekar}, scenario \cite{schildbach2014scenario} and conformal MPC \cite{lindemann2023safe} as well as supermartingale control barrier functions \cite{cosner2023robust}. Unfortunately, the resulting control performance is generally conservative, and the empirical safety of the controllers significantly higher than certified by the theoretical bounds. 

The second category of controllers, typically analyzed through the lens of Markov Decision Processes (MDPs), directly considers the distribution of the system trajectory to synthesize optimal controllers. While these methods provide tight statistical guarantees, they are computationally expensive and restricted to low-dimensional systems \cite{abate2008probabilistic, schmid2022probabilistic,miller2024unsafe}. While scalable heuristics exist \cite{thorpe2019model, bansal2021deepreach}, these approaches lack theoretical guarantees, which are highly desirable for safety-critical applications. 

The goal of this work is to maximize the reach-avoid probability of a stochastic system without resorting to conservative statistical bounds, while preserving scalability for a suitable class of systems. The core idea is to analyse the closed-loop behaviour of a stochastic system controlled by an MPC. We model the closed-loop system as an MDP in which reference signals are sent to the MPC, and exploit the flexibility of DP to generate those reference signals that yield a maximum reach-avoid probability.


Control architectures for stochastic systems remain relatively unexplored. A risk-aware architecture is proposed in \cite{akella2024risk}, which empirically achieves a safe control behaviour under disturbances while preserving scalability. In contrast, this work focuses on providing formal reach-avoid certificates.  

\textbf{Problem Formulation.} 
We design control hierarchies for linear stochastic systems with state space $\mathcal{X}=\mathbb{R}^n$, polyhedral input space $\mathcal{U}\in\mathcal{B}(\mathbb{R}^m)$ and dynamics
\begin{align}
    \label{eq_ct_lti_system}
\dot x(t) &= \underbrace{\footnotesize\begin{bmatrix}
        A_{c,1} & A_{c,2} \\ 0 & A_{c,4}
    \end{bmatrix}}_{A_c}\underbrace{\footnotesize\begin{bmatrix}
        x^{\text{s}}(t) \\ x^{\text{d}}(t)
    \end{bmatrix}}_{x(t)} +\underbrace{\footnotesize\begin{bmatrix}
        B_{c,1} \\ B_{c,2}
    \end{bmatrix}}_{B_c}u(t)+\underbrace{\footnotesize\begin{bmatrix}
        E_{c,1} \\ 0
    \end{bmatrix}}_{E_c}w(t),
\end{align}
where $w(\cdot)$ is an i.i.d. noise process and $(A_c,B_c)$ stabilizable. We explicitly discriminate between deterministically evolving states $x^{\text{d}}\in\mathbb{R}^{n_{\text{d}}}$ and probabilistically evolving states $x^{\text{s}}\in\mathbb{R}^{n_{\text{s}}}$; we will later exploit deterministically evolving states to reduce our algorithm's complexity. Let $s,\tau:\mathbb{R}^{n_{\text{s}}}\rightarrow\mathbb{R}$, $h_i^{\text{d}}\in\mathbb{R}^{n_{\text{d}}}$, $\mathbf{0}\in\mathbb{R}^{n_{\text{s}}}$ denote a zero-vector, and consider the sets
\begin{align*}
    &\mathcal{G}&&
    \hspace{-1em}=\left\{\begin{bmatrix}
    x^{\text{s}} \\ x^{\text{d}} 
\end{bmatrix}\in\mathcal{X}\; \middle|\;\underbrace{\begin{bmatrix}
    \mathbf{0} & h_i^{\text{d}}
\end{bmatrix}}_{h_i}\begin{bmatrix}
    x^{\text{s}} \\ x^{\text{d}} 
\end{bmatrix}\leq b \ \forall i\in[M_h]\right\}, \\
&\mathcal{S}&&\hspace{-1em}=\left\{\begin{bmatrix}
    x^{\text{s}} \\ x^{\text{d}} 
\end{bmatrix}\in\mathcal{G}\; \middle|\;s(x_s)\leq 0\right\}, 
\\
&\mathcal{T}&&\hspace{-1em}=\left\{\begin{bmatrix}
    x^{\text{s}} \\ x^{\text{d}} 
\end{bmatrix}\in\mathcal{G}\; \middle|\;\tau(x_s)\leq 0\right\}.
\end{align*}
The set $\mathcal{G}$ consists of $M_h\in\mathbb{N}$ half-space constraints on $x^{\text{d}}$ that must hold deterministically. In contrast, the safe set $\mathcal{S}\in\mathcal{B}(\mathcal{X})$ and target set $\mathcal{T}\in\mathcal{B}(\mathcal{S})$, which we assume to be compact, define a reach-avoid specification on $x^{\text{s}}$. Since $x^{\text{s}}$ evolves probabilistically, the specification may only be satisfied up to a probability. Maximizing this probability over a finite time horizon leads to our problem formulation.


\refstepcounter{problemcounter}
\begin{mdframed}
Find a controller $\pi^{\star}:\mathcal{X}\rightarrow\mathcal{U}$ solving
    \begin{subequations}
        \label{prob_reach_avoid}
        \begin{align}
            &\!{\max}_{\pi} && \ \mathbb{P}(\exists \tau\!\in\![0,T] \ | \  x([0,\tau))\!\in\!\mathcal{S}, x(\tau)\!\in\!\mathcal{T}) 
             \\
            & \text{s.t.} && \ \dot x(t) = A_cx(t) + B_cu(t) + E_cw(t)
            \\
            &&& \ x(t)\!\in\!\mathcal{G}, \  u(t)\!\in\!\mathcal{U}, \ u(t)\!=\!\pi(x(t))  \\ &&& \ \forall t\!\in\![0,T].
        \end{align}
    \end{subequations}
\end{mdframed}
\textbf{Running Example.} To illustrate our problem formulation, consider a quadcopter with position $p=\begin{bmatrix}
    p_x & p_y & p_z
\end{bmatrix}^{\top}$ and attitude $\theta\in\mathbb{R}^3$ (roll, pitch, yaw), leading to the state $x = \begin{pmatrix} p^\top & \dot{p}^\top & \theta^\top & \dot{\theta}^\top \end{pmatrix}^\top\in\mathbb{R}^{12}$, inputs described by the rotor thrust force and moments for roll, pitch and yaw leading to $u\in\mathbb{R}^4$, and linear dynamics \cite{schmid2022real}
\begin{align}
    \dot{x}(t) &= Ax(t) + Bu(t) + \begin{bmatrix} 
    1 & 0 & 0 & \dots & 0\\
    0 & 1 & 0 & \dots & 0
    \end{bmatrix}^{\top}w(t),
    \label{eq_linear_model_quadcopter}
\end{align}
where $w\in\mathbb{R}^2$ is an i.i.d. noise affecting the $x$- and $y$-position. Define ${x}^{\text{s}}=\begin{bmatrix}
    p_x & p_y 
\end{bmatrix}$ and ${x}^{\text{d}}=\begin{bmatrix}
     p_z & \dot p^{\top} & \theta^{\top} & \dot \theta^{\top}
\end{bmatrix}^{\top}$. The set ${\mathcal{G}}\subseteq\mathbb{R}^{12}$ encodes safety limits on the attitude, velocity and height that must be respected deterministically. The sets ${\mathcal{S}}\subseteq\mathbb{R}^{12}$ (safe, obstacle-free space) and ${\mathcal{T}}\subseteq\mathbb{R}^{12}$ (target) impose a reach-avoid specification on the x-y-position of the quadcopter which must be satisfied with maximum probability. 


\textbf{Contributions.} We summarize our key results as follows.
\begin{itemize}%
    \item We approach Problem \ref{prob_reach_avoid} via control architectures combining MPC with a DP reference generator. The MPC allows one to reduce the dimensionality of the DP planning space from the state space dimension $n$ to the dimensionality of the probabilistically evolving state component $n_{\text{s}}$.
    \item We robustify the framework to model abstractions necessary for DP on systems with continuous state spaces.
    \item We empirically demonstrate the scalability and flexibility of our method on a linear perturbed quadcopter model in cluttered reach-avoid environments.
\end{itemize}

After introducing the most elemental concepts of MPC and MDPs in Section \ref{sec_back}, we present our layered approach in Section \ref{sec_hierarchy}, followed by a robustification to model abstractions customary for DP on systems with continuous state spaces in Section \ref{sec_abstraction}. We numerically evaluate our approach in Section \ref{sec_numerical_example} and summarize our findings in Section \ref{sec_conclusion}.

\textbf{Mathematical Notation}
We denote by $[N]$ the set $\{0,1,\dots,N-1\},N\in\mathbb{N}$.  Let $\mathcal{A}, \mathcal{X}$ be two sets. The indicator function is defined as $\mathds{1}_{\mathcal{A}}(x)=1$ if $x\in {\mathcal{A}}$ and $\mathds{1}_{\mathcal{A}}(x)=0$ otherwise. We define the Minkowski sum of two sets as $\mathcal{A} \oplus \mathcal{X} = \{ a + x \mid a \in \mathcal{A}, x \in \mathcal{X} \}$ and the Pontryagin set difference as $\mathcal{A} \ominus \mathcal{X} = \{ a \mid a + x \in \mathcal{A} \text{ for all } x \in \mathcal{X} \}$. 
Further, let $\mathcal{B}(A)$ be the Borel $\sigma$-algebra generated by $A$, $\mathbb{P}$ and $\mathbb{E}$ the probability and expectation operators, respectively. The projection operator $\text{proj}_{\text{d}}(x)=x^{\text{d}}$ projects the state $x\in\mathcal{X}$ of system \eqref{eq_ct_lti_system} onto its deterministic state component. Finally, $B_r=\{x\in\mathbb{R}^n\ | \ \|x\|\leq r\}$ for some context-dependent $n\in\mathbb{N}$. 

\section{Background}
\label{sec_back}
\textbf{Model Predictive Control.}
Consider the system $x_{j+1} = f(x_j,u_j)$, $j\in\mathbb{N}$, with state and input constraints $\mathcal{X}$, $\mathcal{U}$, respectively, and let $\mathcal{X}_J$ be a control invariant constraint admissible set, i.e., for all $x_j\in\mathcal{X}_J$ there exist $u_j,u_{j+1},\dots\in\mathcal{U}$ such that $x_{j+1},x_{j+2}\dots\in\mathcal{X}_J$. Let $c_0,\dots,c_{J-1}:\mathcal{X}\times\mathcal{U}\rightarrow\mathbb{R}$, $c_J:\mathcal{X}\rightarrow\mathbb{R}$ denote some cost objective and $J\in\mathbb{N}$ the prediction horizon. At every time-step $j\in\mathbb{N}$, the MPC solves
\begin{subequations}
    \label{eq_basic_SMPC}
    \begin{align}
        &\min_{\{z_l\}_{l=0}^J, \{v_l\}_{l=0}^{J-1}} && c_J(z_J) + \sum_{l=0}^{J-1} c_l(z_l,v_l) \\
        &\text{s.t.} && z_0 = x_j, \quad  z_{l+1} = f(z_l, v_l) \\
        &&& z_l\in\mathcal{X},\quad z_J \in\mathcal{X}_J, \quad v_l\in\mathcal{U},\\
        &&& \forall \ l\in[J],
    \end{align}
\end{subequations}
and applies the minimizing $u_j=v_0$ to the system. Assuming that \eqref{eq_basic_SMPC} is feasible from $x_0$ at $j=0$, it remains feasible for all $j> 0$, which implies $x_j\in\mathcal{X}$, $u_j\in\mathcal{U}$ for all $j\geq 0$ \cite{rawlings2020model}.

\textbf{Markov Decision Processes.}
We define an MDP over a finite time horizon $N$ as a tuple $\mathcal{M}=(\mathcal{X}, \mathcal{U}, \mathcal{Q}, \ell_{0},\dots,\ell_{N})$, where the state space $\mathcal{X}$ and the input space $\mathcal{U}$ are Borel subsets of complete separable metric spaces equipped with the $\sigma$-algebras $\mathcal{B}(\mathcal{X})$ and $\mathcal{B}(\mathcal{U})$, respectively. Given a state $x_k\!\in\!\mathcal{X}$ and an input $u_k\!\in\!\mathcal{U}$, the stochastic kernel $\mathcal{Q}\!:\!\mathcal{B}(\mathcal{X})\!\times\!\mathcal{X}\!\times\!\mathcal{U}\!\rightarrow\![0,1]$ describes the stochastic state evolution $x_{k+1}\sim \mathcal{Q}(\cdot|x_k,u_k)$.
Additionally, $\ell_k\!:\!\mathcal{X}\!\times\mathcal{U}\!\rightarrow\!\mathbb{R}_{\geq 0}, k\in[N-1]$ and $\ell_N\!:\!\mathcal{X}\!\rightarrow\!\mathbb{R}_{\geq 0}$ denote measurable, non-negative functions called stage and terminal costs. A deterministic Markov policy is a sequence $\pi=(\mu_0,\dots,\mu_{N-1})$ of measurable maps $\mu_k\!:\!\mathcal{X}\!\rightarrow\!\mathcal{U}, k\in [N-1]$. We denote the set of deterministic Markov policies by $\Pi$. 


\section{Architectures for Linear Stochastic Systems}
\label{sec_hierarchy}
We approach Problem \ref{prob_reach_avoid} via a layered control architecture. The architecture consists of an MPC which tracks reference signals generated by a DP-based state feedback policy, see Fig.~\ref{fig_control_hierarchy}. 

\begin{figure}
    \centering
    \begin{tikzpicture}[
        block/.style={rectangle, draw, thick, text centered, rounded corners, minimum height=.7cm},
        label_style/.style={text width=8em, text centered, font=\small},
        arrow_style/.style={-latex, thick},
        ]
    
        \node (outer) [block, text width=2cm] {DP Policy};
        \node (inner) [block, text width=2cm, right=.5cm of outer] {MPC};
        \node (system) [block, text width=2cm, right=.5cm of inner] {System};
    
        
        \draw [arrow_style] (outer) -> node[above, midway] {$a$} (inner);
        \draw [arrow_style] (inner) -- node[above, midway] {$u$} (system);
        
        
       
        \draw [arrow_style] ([xshift=0.2cm]system.east) -- ([xshift=0.2cm,yshift=-0.8cm]system.east) -| (inner.south);
        
        \node (branch) [above right=0cm of system.east] {$x$};
        
        \draw [arrow_style] (system.east) --([xshift=0.2cm]system.east) -- ([xshift=0.2cm,yshift=-0.8cm]system.east) -| (outer.south);
        
        \node (branch) [below=0.37cm of inner.south, draw, fill, circle, inner sep=1pt] {};
    \end{tikzpicture}
    \caption{The MPC receives commands $a$, which contain a reference trajectory, and computes inputs $u$ based on the system state $x$. It ensures that ${x}(t)\in{\mathcal{G}}$ and $u(t)\in\mathcal{U}$ for all $t\geq 0$. The DP policy generates commands $a$ at time intervals of $\Delta_t$ based on $x(k\Delta_t)$, $k\in\mathbb{N}$. The command $a$ influences the behavior of the MPC, and it is chosen such that the closed-loop system achieves a maximum reach-avoid probability.}
    \label{fig_control_hierarchy}
\end{figure}
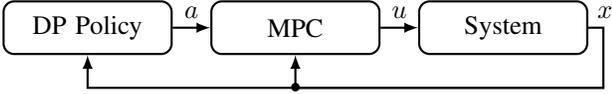

\subsection{Model Predictive Reference Tracking Controller}
Let $\Delta_t=T/N$ and $\delta_t=\Delta_t/J$ for some $N,J\in\mathbb{N}$. We denote the state at sub-time-steps as $x_{k,j}=x(k\Delta_t + j\delta_t)$, $k\in[N],j\in[J]$ and discretize system \eqref{eq_ct_lti_system} with a sampling-time of $\delta_t$,
\begin{align}
    x_{k,j+1} &=
    \underbrace{\begin{bmatrix}
        A_{1} & A_{2} \\ 0 & A_{4}
    \end{bmatrix}}_{A}\underbrace{\begin{bmatrix}
        x^{\text{s}}_{k,j} \\ x^{\text{d}}_{k,j}
    \end{bmatrix}}_{x_{k,j}}\!+\!\underbrace{\begin{bmatrix}
        B_{1} \\ B_{2}
    \end{bmatrix}}_{B}u_{k,j} + \underbrace{\begin{bmatrix}
        E_{1} \\ 0
    \end{bmatrix}}_{E}w_{k,j},
\label{eq_discrete_time_lti_system}
\end{align} 
where $x_{k,j}^{\text{d}}$ propagates according to the noise-free dynamics
\begin{align}
    \label{eq_discrete_time_lti_system_deterministic_subsystem}
    x^{\text{d}}_{k,j+1}=A_4x^{\text{d}}_{k,j}+B_2u^{\text{d}}_{k,j}.
\end{align}
Given tightened constraint sets $\hat{\mathcal{G}}, \hat{\mathcal{G}}_{J}\subseteq \mathcal{G}$, an initial state ${x}_{k,0}\in\hat{\mathcal{G}}_{J}$ for some $k\in[N]$, weighting matrices $Q\in\mathbb{R}^{n\times n}, Q\succeq 0$, $R\in\mathbb{R}^{m\times m}, R\succ 0$, and a reference sequence ${x}_{0}^{\text{ref}},\dots,{x}_{J}^{\text{ref}}\in\mathbb{R}^n$, the MPC solves
\begin{subequations}
    \makeatletter
    \def\@currentlabel{MPC}
    \makeatother
    \label{eq_smpc_inner_layer}
    \renewcommand{\theequation}{MPC.\arabic{equation}} 
\begin{align}
    &\min_{\{z_l\}_{l=j}^J,  \{v_l\}_{l=j}^{J-1}} && \sum_{l=j}^{J} \|z_l-x_{l}^{\text{ref}}\|_{Q} + \sum_{l=j}^{J-1}\|v_l\|_R \\
    &\text{s.t.} \  && z_{j}\!=\!x_{k,j}, \ \ z_{l+1}\!=\!Az_l\!+\!Bv_l, \label{eq_smpc_inner_layer_initial_state_constraint}\\
    &&& {z}_l\in\hat{\mathcal{G}}, \ \ {z}_{J} \in \hat{\mathcal{G}}_{J}, \ \ v_l\in\mathcal{U}, \label{SMPC_input_constraint}
    \\&&& \forall \ l=j\dots,J-1,
\end{align}
\end{subequations}
at every time-step $j\in[J]$ and applies the optimal input $u(t)=v_j$ to the system over the interval $t\in[k\Delta_t + j\delta_t, k\Delta_t + (j+1)\delta_t]$. 
While \eqref{eq_smpc_inner_layer} can only enforce constraints at discrete time steps, Problem \ref{prob_reach_avoid} asks for constraint satisfaction over the entire continuous duration $[0, T]$. To bridge the two, the following lemma demonstrates how the original constraint $\mathcal{G}$ can be tightened to $\hat{\mathcal{G}}$ such that discrete-time satisfaction of $\hat{\mathcal{G}}$ ensures continuous-time satisfaction of $\mathcal{G}$.
\begin{lemma}[Proof in Appendix \ref{app_proof_lem_cont_time_constraint_satisfaction}]\label{lem_cont_time_constraint_satisfaction}
    Let $\kappa_x=\max_{x^{\text{d}}\in\mathbb{R}^{n_{\text{d}}}}\|x^{\text{d}}\|$ subject to $h_i^{\text{d}}x^{\text{d}}\leq b$ for all $i\in[M_h]$, $\kappa_u=\max_{u\in\mathcal{U}}\|B_2u\|$, and
    \begin{align*}
        \eta\!=\!\begin{cases}
            (e^{\|\!A_4\!\|\delta_t}\!-\!1)\kappa_x\!+\!\frac{1}{\|\!A_4\!\|}(e^{\|\!A_4\!\|\delta_t}\!-\!1) \kappa_u & \text{if $\|A_4\|\!>\!0$} \\
            \delta_t \kappa_u & \text{otherwise.}
        \end{cases}
    \end{align*}
    Then, if $x_{k,j}\in\hat{\mathcal{G}}=\mathcal{G}\ominus B_{\eta}$,
    it follows that $x(t)\in\mathcal{G}$ for all $t\in[k\Delta_t+j\delta_t,k\Delta_t+(j+1)\delta_t]$.
\end{lemma}
The scalar $\eta$ bounds the maximum change of $x^{\text{d}}$ within a time-step $j$ under any input $u\in\mathcal{U}$ and state $x_{k,j}\in\mathcal{G}$. In our running example, $\eta$ bounds the change of the quadcopter's $z$-position, velocities, attitude and attitude rates over a time-duration of $\delta_t$. 

Let $\hat{\mathcal{G}}$ be defined as in Lemma \ref{lem_cont_time_constraint_satisfaction}. Then, by definition of $\mathcal{G}$, $\hat{\mathcal{G}}=\mathbb{R}^{n_{\text{s}}}\times \hat{\mathcal{G}}^{\text{d}}$ with $\hat{\mathcal{G}}^{\text{d}}=\{x^{\text{d}}\in\mathbb{R}^{n_{\text{d}}}\ | \ h_i^{\text{d}}x^{\text{d}}+\|h_i^{\text{d}}\|\eta\leq b \ \forall i\in[M_h]\}$. Let the set $\hat{\mathcal{G}}_J^{\text{d}}$ denote a constraint-admissible control invariant set for the deterministic subsystem \eqref{eq_discrete_time_lti_system_deterministic_subsystem} subject to constraints $x_{k,j}^{\text{d}}\in\hat{\mathcal{G}}^{\text{d}}$, $u_{k,j}\in\mathcal{U}$ and denote the extension to the space $\mathcal{X}$ as $\hat{\mathcal{G}}_J=\mathbb{R}^{n_{\text{s}}}\times \hat{\mathcal{G}}_J^{\text{d}}$. Then, we show that starting from $x(0)\in\hat{\mathcal{G}}_J$, continuous-time constraint satisfaction and recursive feasibility of \eqref{eq_smpc_inner_layer} is guaranteed.
\begin{proposition}
    [Proof in Appendix \ref{app_proof_prop_recursive_feasibility_of_MPC}]\label{prop_recursive_feasibility_of_MPC}
    If $x(0)\in\hat{\mathcal{G}}_J$, then \eqref{eq_smpc_inner_layer} is feasible for all $j\in[J]$ and $k\in[N]$. Further, system \eqref{eq_ct_lti_system} controlled by \eqref{eq_smpc_inner_layer} satisfies ${x}(t)\in{\mathcal{G}}$ and $u(t)\in\mathcal{U}$ for all $t\in [0, T]$. Lastly, ${x}(k\Delta_t)\in\hat{\mathcal{G}}_J$ for all $k\in [N+1]$. 
\end{proposition}
The proof shows recursive feasibility through the terminal set $\hat{\mathcal{G}}_J$ and the fact that \eqref{eq_smpc_inner_layer} only imposes constraints on $x^{\text{d}}$. Combined with Lemma \ref{lem_cont_time_constraint_satisfaction} the constraint $\mathcal{G}$ is satisfied in continuous time.

\subsection{Dynamic Programming based Reference Generator}
\label{sec_arch_outer}
In the spirit of our layered architecture, we do not impose the constraints $\mathcal{S}$ and $\mathcal{T}$ within \eqref{eq_smpc_inner_layer}. Instead, a DP-based state feedback policy must generate reference signals that guide \eqref{eq_smpc_inner_layer} to a maximum reach-avoid probability. The DP-based reference generator is executed at fixed time-intervals $\Delta_t$ and acts as a discrete time state-feedback controller on ${\mathcal{X}}$. Given $x_k={x}(k\Delta_t)$, $k\in[N]$, it returns a reference sequence $a_k=(x_{0}^{\text{ref}},\dots,x_{J}^{\text{ref}})\in\mathcal{A}\subseteq \mathcal{X}^{J+1}$ chosen from a finite set $\mathcal{A}$. We refer to the actions $a_k$ as commands. The goal is to choose those commands that maximize the reach-avoid probability. 

Note that \eqref{eq_smpc_inner_layer} is a state-feedback controller. Therefore, its optimal solution at time-step $k$ only depends on the state $x_k$. Further, the dynamics \eqref{eq_ct_lti_system} are Markov, which renders the closed-loop system a discrete time Markov Process with time-steps $j$. The choice of commands $a_k$ turns the closed-loop system into an MDP with state $x_k$, input $a_k$ and time-steps $k\in[N]$. This enables to use DP to optimally select $a_k$ as a function of $x_k$. 
    
To this end, we model all states in $\mathcal{S}^c\cup \mathcal{T}$ as absorbing, i.e., any state trajectory satisfying $x(t_0)\in\mathcal{S}^c\cup \mathcal{T}$ for some $t_0\in[k\Delta_t,(k+1)\Delta_t]$ remains fixed at $x(t)=x(t_0)$ for all $t\geq t_0$. We further assume the following.
\begin{assumption}[Existence of a Kernel]\label{ass_existence_of_kernel}
    System \eqref{eq_ct_lti_system} in closed-loop with \ref{eq_smpc_inner_layer} yields a discrete time transition kernel $\mathcal{Q}:\mathcal{B}(\mathcal{X})\times \mathcal{X}\times\mathcal{A}\rightarrow[0,1]$, such that $\mathcal{Q}(C|{x}_k,a_k) = \mathbb{P}({x}_{k+1}\in C|{x}_k,a_k)$ for all $k=0,\dots,N-1$, ${x}_k\in\hat{\mathcal{G}}_J$, $a_k\in\mathcal{A}$ and $C\in\mathcal{B}({\mathcal{X}})$. 
\end{assumption}
We thereby define $\mathcal{Q}(x_{k+1}=x_k|{x}_k,\cdot)=1$ and zero otherwise for any $x_k\in\hat{\mathcal{G}}_J^c$. Note that this, as well as $\mathcal{S}$ and $\mathcal{T}$ being absorbing, are modeling choices for the MDP, which do not necessarily reflect the dynamics of the physical system. 

Under Assumption \ref{ass_existence_of_kernel}, we summarize the closed-loop system as the MDP $\mathcal{M}=(\mathcal{X} ,\mathcal{A},\mathcal{Q},\ell_0,\dots,\ell_N)$. 
Then, choosing $\ell_k(\cdot,\cdot)=0$ and $\ell_N(x_N)=\mathds{1}_{\mathcal{T}}(x_N)$, any policy that maximizes $\mathbb{E}[\ell_N({x}_N)+\sum_{k=0}^{N-1}\ell_k({x}_k,u_k)]$ also maximizes the reach-avoid probability in $\mathcal{M}$. 

Using DP, we recursively compute value functions $V_k({x}_k)=\sup_{\pi\in\Pi}\mathbb{P}(x_N\in\mathcal{T}|{x}_k,\pi)$. Note that, since $\mathcal{T}\cup\mathcal{S}^c$ is absorbing, for any $k\in[N]$ $V_k({x}_k)=1$ if $x_k\in\mathcal{T}$ and $V_k({x}_k)=0$ if $x_k\in\mathcal{S}^c$. This restricts analysis to the compact set ${x}_k\in{\mathcal{S}}\setminus{\mathcal{T}}$. For all ${x}_k, {x}_N\in{\mathcal{S}}\setminus{\mathcal{T}}$, $k\in[N]$, let
\begin{subequations}
    \label{eq_true_reach_avoid_probabilities}
    \begin{align}
    V_N({x}_N) &= 0, \\
    V_k({x}_k) &= \max_{a_k\in\mathcal{A}} \!\int_{{\mathcal{X}}}\hspace{-.3em}V_{k+1}({x}_{k+1})\mathcal{Q}(d{x}_{k+1}|{x}_k,a_k),
\end{align}
\end{subequations}
where a maximizing argument $\pi_k({x}_k)\in{\arg\max}_{a_k\in\mathcal{A}} \int_{{\mathcal{X}}}V_{k+1}({x}_{k+1})\mathcal{Q}(d{x}_{k+1}|{x}_k,a_k)$ exists for every ${x}_k\in{\mathcal{S}\setminus\mathcal{T}}$ by finiteness of $\mathcal{A}$ \cite{schmid2022probabilistic}. Generating references $a_k=\pi_k(x_k)$ for \eqref{eq_smpc_inner_layer} at every time-step $k\in[N]$ then yields a reach-avoid probability of $V_0(x_0)$ \cite{abate2008probabilistic}. 

\section{Abstraction and Error Bounds}
\label{sec_abstraction}
Since the state space ${\mathcal{X}}$ is continuous, the minimization in \eqref{eq_true_reach_avoid_probabilities} involves an infinite number of states. To render the problem computationally tractable, we approximate $\mathcal{M}$ via a finite gridding abstraction \cite{abate2010approximate}. 

\subsection{Gridding Abstraction}
To construct a gridding abstraction of $\mathcal{M}$, let $\mathcal{H}=\{{x}\in{\mathcal{X}}\ | \ \|{x}\|_{\infty}\leq \frac{1}{2}\zeta\}$ be a closed hypercube for some $\zeta>0$. We cover $\mathcal{X}$ using a tiled grid of $M_x$ (possibly infinite) disjoint subsets $\mathcal{H}_0,\dots,\mathcal{H}_{M_x}\in \mathcal{B}(\mathcal{H})$ translated by the centers $c_i\in\mathcal{X}$ for $i\in[M_x]$, respectively. We denote the collection of all centers as $\mathcal{C}=\{{c}_i\in\mathcal{H}_i\}_{i=1}^{M_x}$ and define $x(0)\in\mathcal{H}_0$.

The gridding abstraction approximates all states ${x}\in\mathcal{H}_i$ by their corresponding center ${x}={c}_i$, $i=1,\dots,M_x$, reducing the continuous set $\mathcal{X}$ to the countable set of states $\mathcal{C}$. Consider the MDP $\hat{\mathcal{M}}=(\mathcal{C},\mathcal{A},\hat{\mathcal{Q}},{\ell}_0,\dots,{\ell}_N)$, where the cost functions $\ell_0,\dots,\ell_{N}$ remain as before. Given some $i,j\in[M_x]$, the transition kernel is approximated as $\hat{\mathcal{Q}}({c}_i|x,a)=\mathcal{Q}(\mathcal{H}_i|c_j,a)$ for all $x\in\mathcal{H}_j$. It can be empirically estimated via Monte-Carlo simulations of the MPC-controlled system from each $x\in\mathcal{C}$ and $a\in\mathcal{A}$ over a time-duration of $\Delta_t$. A policy $\pi:\mathcal{C}\rightarrow\mathcal{A}$ computed on $\hat{\mathcal{M}}$ is executed on $\mathcal{M}$ by applying $a_k=\pi_k({c}_i)$ whenever ${x}_k\in\mathcal{H}_i$. 

Unfortunately, if for any $i,j\in[M_x]$ the probability $\mathcal{Q}(\mathcal{H}_i|x,a)$ varies significantly across ${x}\in\mathcal{H}_j$, the approximation ${\mathcal{Q}}(\mathcal{H}_i|{c}_j,a)\approx\mathcal{Q}(\mathcal{H}_i|x,a)$ becomes inaccurate. As a consequence, the reach-avoid probability maximized by $\pi$ in $\hat{\mathcal{M}}$ may not be achieved when executing $\pi$ on $\mathcal{M}$. 

To overcome this issue, we robustify the architecture to errors incurred by the gridding abstraction. Consider some $x(t_0)\in\mathcal{H}_i$ and $x'(t_0)=c_i$ which evolve via the dynamics \eqref{eq_ct_lti_system} under the same noise $w(\cdot)$ and input $u(\cdot)$. Then, the tube $\mathcal{E}(t-t_0)=\{e^{A_c(t-t_0)}x\ | \ x\in\mathcal{H}\}$ contains the deviation of $x(t)$ from $x'(t)$,  
\begin{align}
    \label{eq_tube_bound}
    x(t)-x'(t)\in\mathcal{E}(t-t_0)
\end{align}for all $t\geq t_0$. To simplify the upcoming discussion, we over-approximate the union of $\mathcal{E}(t)$ over an interval $\Delta_t$ using a time-invariant ball $B_r$ of radius $r\geq 0$.
\begin{lemma}[Proof in Appendix \ref{proof_lem_B_bounds_Epsilon}]\label{lem_B_bounds_Epsilon}
    Let $r=e^{\|A_c\|\Delta_t}\max_{{x}\in\mathcal{H}}\|{x}\|$, then $\bigcup_{t\in[0,\Delta_t]}\mathcal{E}(t)\subseteq B_r$. 
\end{lemma}
In terms of the quadcopter example, given the same noise $w(\cdot)$ and input $u(\cdot)$, this means that starting from two states with maximum distance $\frac{1}{2}\zeta$, e.g., different velocities, yields trajectories that deviate from each other no further than $r$ over a time-step $k$.

While the abstraction only considers trajectories $x(t)$ evolving from the centers $c_i\in\mathcal{C}$, $i\in[M_x]$, if the constraints are strictly satisfied with a distance $r$, they are also satisfied for trajectories evolving from anywhere within $\mathcal{H}_i$.


\subsection{Robust Model Predictive Reference Tracking Controller}
Eq.~\eqref{eq_tube_bound} only holds if the two trajectories $x(t)$, $x'(t)$ evolve under the same noise $w(\cdot)$ and input $u(\cdot)$. Consequently, if we want to exploit this result to robustify our control architecture, we need to modify \eqref{eq_smpc_inner_layer} such that under a given command $a_k=({x}^{\text{ref}}_0,\dots,{x}^{\text{ref}}_J)$ and fixed noise sequence $w(\cdot)$ the same input sequence $u_{k,0},\dots,u_{k,J-1}$ is generated for every $x_k\in\mathcal{H}_i$ for each $i\in[M_x]$. In place of \eqref{eq_smpc_inner_layer}, this is achieved by the robust MPC%
\begin{subequations}%
    \makeatletter%
    \def\@currentlabel{RMPC}%
    \makeatother%
    \label{eq_smpc_inner_layer_robust}%
    \renewcommand{\theequation}{RMPC.\arabic{equation}}%
\begin{align}
    &&\min_{\{z_l\}_{l=j}^J,  \{u_l\}_{l=j}^{J-1}} & \sum_{l=j}^{J} \|z_l\!-\!x_{l}^{\text{ref}}\|_{ Q}\!+\!\sum_{l=j}^{J-1}\|u_l\|_R \\
    &&\text{s.t.} \  & z_j = x_{k,j} - A^j(x_{k,0}-c_i) \label{eq_smpc_inner_layer_initial_state_constraint_robust}\\
    &&& z_{l+1} = Az_l + Bu_l \\
    &&& {z}_{l} \in \Tilde{\mathcal{G}}\label{smpc_terminal_state_constraint}, z_{J}\in \Tilde{\mathcal{G}}_J\ominus\mathcal{H}\ominus B_r 
    \\ &&& u_l\in\mathcal{U}
    \\&&& \forall \ l=j,\dots,J-1.
\end{align}
\end{subequations}
The state constraint \eqref{eq_smpc_inner_layer_initial_state_constraint_robust} initializes $z_j$ as if the trajectory would have evolved from the state $x_{k,0}=c_i$. Using the following result, we design the constraints $\Tilde{\mathcal{G}}$ and $\Tilde{\mathcal{G}}_J$ to ensure that $x_{k,j}\in\hat{\mathcal{G}}$ is still satisfied for all $j\in[J]$ despite the modified initial state constraint. 
\definecolor{myBlue}{HTML}{1D71B8}
\definecolor{myRed}{HTML}{BE1622}
\begin{figure}[!tb]
    \centering
    \begin{tikzpicture}
        \node[anchor=south west, inner sep=0] (image) at (0,0) {
            \includegraphics[width=.6\linewidth]{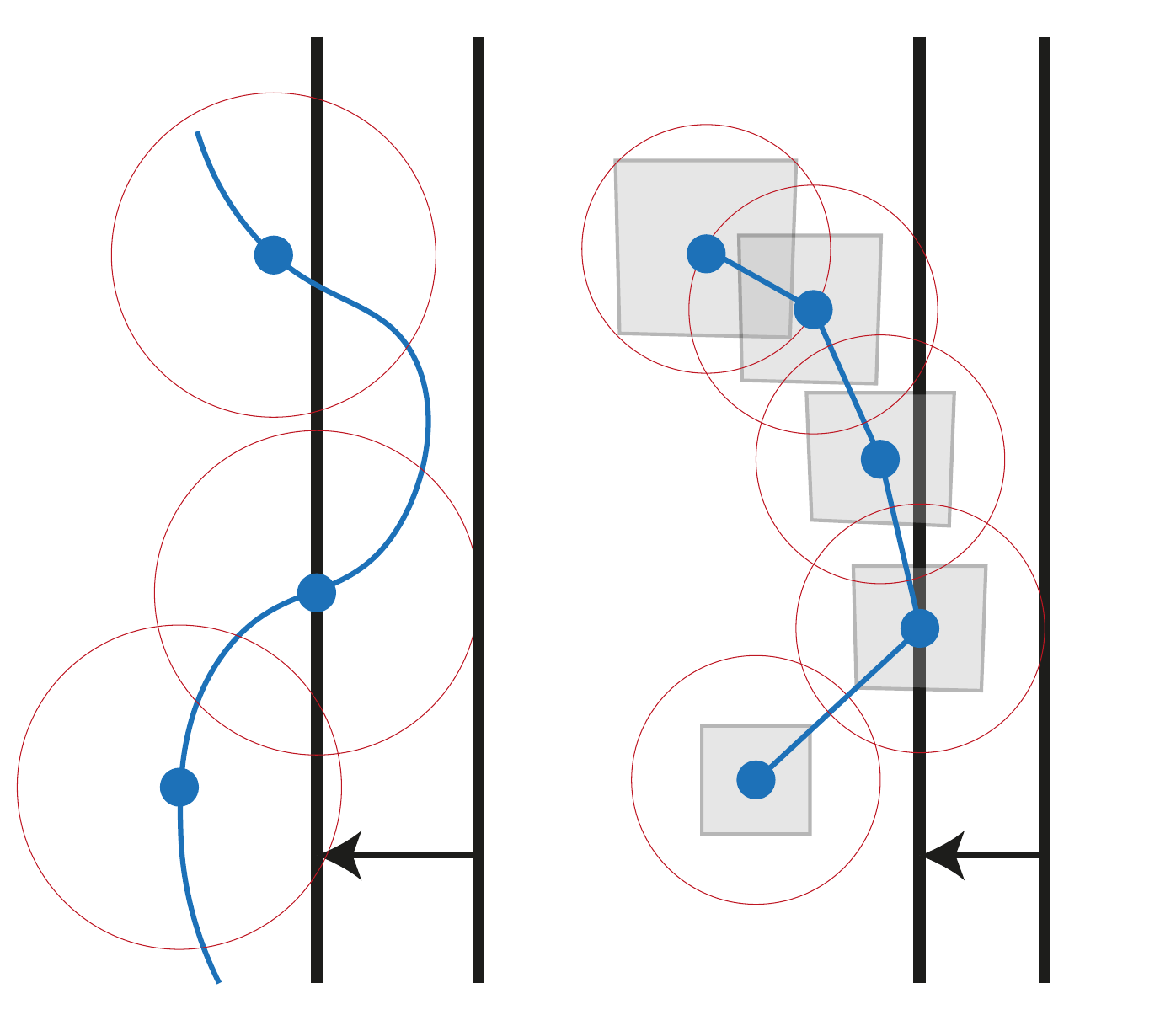}
        };

        \begin{scope}[x={(image.south east)}, y={(image.north west)}]
            
            \node[text=black] at (0.23, 0.05) {$\hat{\mathcal{G}}$};
            \node[text=black] at (0.37, 0.04) {$\mathcal{G}$};
            
            \node[text=black] at (0.35, 0.2) {$\eta$};

            \node[text=myRed] at (0.02, 0.09) {$B_{\eta}$};
            \node[text=myBlue] at (0.12, 0.03) {$x(t)$};
            \node[text=myBlue] at (0.075, 0.23) {$x_{k,j}$};
            \node[text=myBlue] at (0.16, 0.43) {$x_{k,j+1}$};
            
            \node[text=black] at (0.74, 0.05) {$\Tilde{\mathcal{G}}$};
            \node[text=black] at (0.85, 0.05) {$\hat{\mathcal{G}}$};
            
            \node[text=black] at (0.84, 0.2) {$r$};

            \node[text=myRed] at (0.57, 0.09) {$B_r$};
            \node[text=myBlue] at (0.56, 0.235) {$x_{k,j}$};
            \node[text=myBlue] at (0.67, 0.38) {$x_{k,j+1}$};
            
            \node[text=gray] at (0.7, 0.14) {$\mathcal{E}(0)$};
            
        \end{scope}
    \end{tikzpicture}
    \caption{The constraint $x_{k,j}\in\hat{\mathcal{G}}=\mathcal{G}\ominus B_{\eta}$ ensures that $x(t)\in\mathcal{G}$, for all $t\in[k\Delta_t+j\delta_t,k\Delta_t + (j+1)\delta_t]$, $k\in[N]$, $j\in[J]$. The constraint $x_{k,j}\in\Tilde{\mathcal{G}}=\hat{\mathcal{G}}\ominus B_{r}$ ensures that $x_{k,j}\in\hat{\mathcal{G}}$ even when the initial state $x_{k,0}$ is uncertain within a set $\mathcal{H}$. The red circles indicate the balls $B_{\eta}$ and $B_r$, the gray boxes indicate the tube $\mathcal{E}$.}
    \label{fig_sets_G}
\end{figure}

    
    
    
    
    

\begin{corollary}
    \label{lem_if_in_hat_then_in_tilde}
    Let $x(k\Delta_t)\in\mathcal{H}_i$, $x'(k\Delta_t)=c_i$, $k\in[N]$, propagate through \eqref{eq_ct_lti_system} under the same input sequence $u(\cdot)$ and noise sequence $w(\cdot)$. Then, for all $j\in[J]$, $x'(k\Delta_t+j\delta_t)\in\hat{\mathcal{G}}\ominus B_r$ implies $x(k\Delta_t+j\delta_t)\in\hat{\mathcal{G}}$.
\end{corollary}
The corollary is an immediate consequence of Lemma \ref{lem_B_bounds_Epsilon}. Analogously as for $\hat{\mathcal{G}}$, we define $\Tilde{\mathcal{G}}=\mathbb{R}^{n_{\text{s}}}\times \Tilde{\mathcal{G}}^{\text{d}}$ with $\Tilde{\mathcal{G}}^{\text{d}}=\{x^{\text{d}}\in\mathbb{R}^{n_{\text{d}}}\ | \ h_i^{\text{d}}x^{\text{d}}+\|h_i^{\text{d}}\|(\eta+r)\leq b \ \forall i\in[M_h]\}$. Let $\Tilde{\mathcal{G}}_J^{\text{d}}$ denote a constraint admissible control invariant set for the system 
\begin{align}
    \label{eq_discrete_time_lti_system_deterministic_subsystem_robust}
    x^{\text{d}}_{k,j+1}=A_4x^{\text{d}}_{k,j}+B_2u_{k,j}+d_{k,j}^{\text{d}}    
\end{align}
subject to constraints $x_{k,j}\in\Tilde{\mathcal{G}}^{\text{d}}$ and $u_{k,j}\in\mathcal{U}$ for any $\|d^{\text{d}}_{k,j}\|_{\infty}\leq r+ \frac{1}{2}\zeta$. We denote the extension to the space $\mathcal{X}$ as $\Tilde{\mathcal{G}}_J=\mathbb{R}^{n_{\text{s}}}\times \Tilde{\mathcal{G}}_J^{\text{d}}$. Then, if $x(0)\in\Tilde{\mathcal{G}}_J\ominus\mathcal{H}$, our next result ensures that $\Tilde{\mathcal{G}}$ and $ \Tilde{\mathcal{G}}_J$ yield continuous time constraint satisfaction of $\mathcal{G}$ and recursive feasibility for \eqref{eq_smpc_inner_layer_robust}.
\begin{proposition}[Proof in Appendix \ref{proof_prop_RMPC_guarantees}]\label{prop_RMPC_guarantees}
    Let $x(0)\in\Tilde{\mathcal{G}}_J\ominus\mathcal{H}$. Then, \eqref{eq_smpc_inner_layer_robust} is feasible for all $j\in[J]$ and $k\in[N]$. Further, ${x}(t)\in{\mathcal{G}}$, $u(t)\in\mathcal{U}$ for all $t\in [0, T]$, and ${x}((k+1)\Delta_t)\in{\mathcal{G}}_J\ominus\mathcal{H}$ for all $k\in [N]$. Lastly, given any fixed noise realization $w(\cdot)$, $a_k\in\mathcal{A}$, $k\in[N]$ and $i\in[M_x]$, \eqref{eq_smpc_inner_layer_robust} generates the same input sequence $u_{k,0},\dots,u_{k,J-1}$ for all $x_k\in\mathcal{H}_i$.
\end{proposition}
The proof follows analogously to Lemma \ref{lem_cont_time_constraint_satisfaction}. In summary, over a time-step $k$ the initialization strategy in \eqref{eq_smpc_inner_layer_initial_state_constraint_robust} renders the input sequence generated by \eqref{eq_smpc_inner_layer_robust} agnostic to the specific initial state within a partition. While this leads to trajectory predictions in \eqref{eq_smpc_inner_layer_robust} that do not necessarily begin at the current state measurement $x_{k,j}$ for every $j\in[J]$, the difference is bounded through $B_r$ and robustified to via $\Tilde{\mathcal{G}}$ and $\tilde{\mathcal{G}}_J$. Fig.~\ref{fig_sets_G} provides an overview of the different sets $\mathcal{G}$, $\hat{\mathcal{G}}$, $\Tilde{\mathcal{G}}$ and the effect of the respective tightenings. 

Hereafter, by abuse of notation, we denote by $\mathcal{Q}$ and $\hat{\mathcal{Q}}$ the transition probabilities of system \eqref{eq_ct_lti_system} under \eqref{eq_smpc_inner_layer_robust} in place of \eqref{eq_smpc_inner_layer}.




\subsection{Robust Dynamic Programming based Reference Generator}
Finally, we robustify the DP algorithm to abstraction errors and lower-bound the achieved reach-avoid probability. Building up on Proposition \ref{prop_RMPC_guarantees} and Lemma \ref{lem_B_bounds_Epsilon}, under a fixed noise realization $w(\cdot)$ and \eqref{eq_smpc_inner_layer_robust} as tracking controller, any trajectories $x(t)$ and $x'(t)$ with $x(k\Delta_t)\in\mathcal{H}_i$ and $x'(k\Delta_t)=c_i$ satisfy $x(t)-x'(t)\in B_r$ for all $t\in[k\Delta_t,(k+1)\Delta_t]$. Coherently, we define robustified safe and target sets, $\Tilde{\mathcal{S}}$ and $\Tilde{\mathcal{T}}$, as the union of all partitions $\mathcal{H}_i$ that are fully contained within the tightened sets $\mathcal{S}\ominus B_r$ and $\mathcal{T}\ominus B_r$, leading to $\Tilde{\mathcal{S}} = \bigcup \{ \mathcal{H}_i \mid \mathcal{H}_i \subseteq \mathcal{S} \ominus B_r \}$, and $\Tilde{\mathcal{T}} \subseteq {\mathcal{T}}\ominus B_r$ defined analogously. Consider the transition kernel $\Tilde{\mathcal{Q}}$, which is defined equivalently to $\hat{\mathcal{Q}}$ with $\Tilde{\mathcal{S}}$ and $\Tilde{\mathcal{T}}$ in place of $\mathcal{S}$ and $\mathcal{T}$. Let $\Tilde{\ell}(x_N)=\mathds{1}_{\Tilde{\mathcal{T}}}(x_N)$ for any $x_N\in\mathcal{X}$, $\iota=\{i\in[M_x]\ | \ c_i\in\tilde{\mathcal{S}}\setminus\tilde{\mathcal{T}}\}$ be all centers contained in the tightened safe set, and $\mathcal{C}_j = \{{c}\in\mathcal{C}\ | \ {c}\in\mathcal{H}_j\oplus \mathcal{H}\oplus B_r\}$ be centers in a neighborhood of $c_j$ for any $j\in[M_x]$. For all ${c}_i\in\mathcal{C}$, $k\in[N]$, compute
\newcommand{\midsized}{%
  \fontsize{9.5pt}{11.4pt}\selectfont
}
\begin{midsized}
\begin{subequations}
\begin{align}
    \Tilde{V}_N({c}_i)\!&=\!\Tilde{\ell}_N({c}_i),\label{eq_robust_value_iteration_1} \\
    \Tilde{V}_k({c}_i)\!&=\!\max_{a\in\mathcal{A}} 
    \Tilde{\mathcal{Q}}(\Tilde{\mathcal{T}}| c_i,a)\!+\!\sum_{j\in\iota}\!\min_{c\in\mathcal{C}_j}\!\Tilde{V}_{k+1}({c})\Tilde{\mathcal{Q}}(c_j| c_i,a).\label{eq_robust_value_iteration_2}%
\end{align}%
\label{eq_robust_value_iteration}%
\end{subequations}%
\end{midsized}%
The inner minimization in \eqref{eq_robust_value_iteration_2} leads to an underapproximation of the reach-avoid probability for all $x_k\in\mathcal{H}_i$ by assigning the worst-case value $\Tilde{V}_{k+1}(c)$ within the neighborhood $\mathcal{C}_j$ to any transition from $c_i$ ending in $\mathcal{H}_j$, $j\in[M_x]$. Given $k\in[N]$ and ${c}_i\in\mathcal{C}$, we denote an optimal input in \eqref{eq_robust_value_iteration_2} as $a^{\star}$ and define as optimal policy $\pi_k({c}_i)=a^{\star}$. The policy is executed on $\mathcal{M}$ at every time-step $k\in[N]$ by applying $a_k=\pi_k({c}_i)$ when ${x}_k\in\mathcal{H}_i$. 
 
Let Assumption \ref{ass_existence_of_kernel} hold and ${x}_0\in(\Tilde{\mathcal{G}}_J\ominus\mathcal{H})\cap\mathcal{H}_0$. Consider an optimal policy $\pi$ computed from \eqref{eq_robust_value_iteration} with optimal value $\Tilde{V}_0({c}_0)$. Let $V_0({x}_0)$ denote the reach-avoid probability achieved by executing $\pi$ on $\mathcal{M}$ computed via \eqref{eq_true_reach_avoid_probabilities}. Then, our main result states that $\Tilde{V}_0({c}_0)$ lower bounds the reach-avoid probability achieved when applying the policy $\pi$ to the controller \eqref{eq_smpc_inner_layer_robust} and system \eqref{eq_ct_lti_system}.
\begin{theorem}[Proof in Appendix \ref{app_proof_thm_value_function_underapproximation}]
    \label{thm_value_function_underapproximation}
    $\Tilde{V}_0({c}_0)\!\leq\!V_0({x}_0)$.
\end{theorem}

The overall robustification to the gridding abstraction may be understood as a spatial conservatism on the safe and target set. The MPC uses a tightening by $B_r$, while the value iteration considers worst-case transitions within a range $\mathcal{C}_j$ around every partition $\mathcal{H}_j$ for any $j\in [M_x]$. Both, $B_r$ and $\mathcal{C}_j$ for any $j\in[M_x]$, become arbitrarily small as $\zeta$ approaches zero.

\subsection{Computational Remarks}
\label{sec_computational_remarks}
The DP recursion \eqref{eq_robust_value_iteration} allows for significant simplifications. For any $c\in\mathcal{C}$, $\tilde{V}_k(c)=0$ if $c\in\Tilde{\mathcal{S}}^c$ and $\tilde{V}_k(c)=1$ if $c\in\Tilde{\mathcal{T}}$. This restricts the evaluation of $\tilde{V}$ to states $c\in\mathcal{C}\cap\Tilde{\mathcal{S}}\setminus\tilde{\mathcal{T}}$, which are finitely many by compactness of ${\mathcal{S}}$ and, consequently, of $\tilde{\mathcal{S}}\subseteq\mathcal{S}\ominus B_r$. 

Most importantly, the MPC terminal set allows to reduce the state space dimensionality of the MDP $\Tilde{M}$ without losing any theoretical certificates. 
Let $\hat{\mathcal{G}}_J=\mathbb{R}^{n_{\text{s}}}\times\{ \textbf{0}\}$, where $\textbf{0}\in\mathbb{R}^{n_{\text{d}}}$ is a zero-vector, and choose $\mathcal{H}=\{x\in\hat{\mathcal{G}}_J\ | \ \|x\|_{\infty}\leq \frac{1}{2}\zeta\}$. Then, given any $k\in[N]$, $j\in[J]$, $i\in[M_x]$ and $c_i,x_{k,0}\in\mathcal{H}_i$, since the state component $x^{\text{s}}$ and noise $w(\cdot)$ do not affect the evolution of $x^{\text{d}}$ it holds that $z_j=x_{k,j}-(A^jx_{k,0}-c_i)$ yields $z_j^{\text{d}}=x_{k,j}^{\text{d}}$. Since $x^{\text{s}}$ is unconstrained in \eqref{eq_smpc_inner_layer_robust} $z_{j}\in\hat{\mathcal{G}}$ implies $x_{k,j}\in\hat{\mathcal{G}}$. Consequently, it is sufficient to choose $\Tilde{\mathcal{G}}=\hat{\mathcal{G}}$ to ensure satisfaction of $\hat{\mathcal{G}}$ through \eqref{eq_smpc_inner_layer_robust}. Further, by recursive feasibility of \eqref{eq_smpc_inner_layer_robust}, the terminal constraint $z_J\in\Tilde{\mathcal{G}}_J\ominus\mathcal{H}$ will always force $x^{\text{d}}_{k,J}$ to zero at every time-step $k\in[N+1]$, regardless of the command $a_k$.
As a consequence, $x^{\text{d}}$ will be constantly zero in the MDP and the respective state dimensions neglectable. Of course, this approach comes at a trade-off: While eliminating states from the MDP reduces the computational complexity of the DP algorithm, a more restrictive terminal set potentially degrades the performance of \eqref{eq_smpc_inner_layer_robust}. 

\section{Numerical Examples}
\label{sec_numerical_example}
We empirically validate our proposed control architecture on the linear quadcopter system \eqref{eq_linear_model_quadcopter} using the parameters and constraint set $\mathcal{G}$ (except the position constraints) described in \cite{schmid2022real}. We choose $\delta_t = \SI{0.1}{\second}$, $\Delta_t=\SI{2.5}{\second}$, $T=\SI{100}{\second}$, leading to $J=25$ and $N=40$. The quadcopter is simulated with a time-discretization of $\SI{1}{\milli\second}$, where the disturbance is sampled from $w\sim\mathcal{N}(\begin{bmatrix}
    0 & 0
\end{bmatrix}^{\top},\text{diag($5\cdot10^{-4}$,$5\cdot10^{-4}$)})$ at every simulation time-step. 

The velocity and attitude constraints of the quadcopter are not considered safety-critical; for simplicity, we do not impose continuous time satisfaction of $\mathcal{G}$ and simply choose $\tilde{\mathcal{G
}}=\mathcal{G}$. We exploit that the noise only affects the $x$- and $y$-position and choose $\begin{pmatrix}
    p_z & \dot p^{\top} & \theta^{\top} & \dot \theta^{\top} 
\end{pmatrix}^{\top}\in\mathcal{G}_J^{\text{d}}=\textbf{0}$, where $\textbf{0}\in\mathbb{R}^{10}$ is a zero vector. While this choice requires the quadcopter to attain zero velocity and attitude every $\Delta_t=\SI{2.5}{\second}$, at the respective loss of performance, it allows to reduce the state space of the MDP from twelve to only two dimensions, $p_x$ and $p_y$. Further, as discussed in Section \ref{sec_computational_remarks}, under this choice one may define $\tilde{\mathcal{G}}=\hat{\mathcal{G}}$. The hypercubes are chosen as $\mathcal{H}\subseteq \mathcal{G}_J$ with $\zeta=0.05$, quantizing $p_x$ and $p_y$ into squares of $\SI{0.1}{\meter}\times\SI{0.1}{\meter}$. Since $p_x$ and $p_y$ do not affect the state evolution $\dot x$, we obtain $\mathcal{E}(\cdot)=\mathcal{H}$. 

Instead of full $J$-step reference signals, we generate $|\mathcal{A}|=100$ commands $a=(v_0^x,v_0^y,q_x,q_y)\in\mathbb{R}^4$, where $v_0^x,v_0^y\sim\mathcal{N}(0,0.3)$, $q_x,q_y\sim U(0,1000)$. Position references are then computed by integrating the constant velocity references $v_0^x$ and $v_0^y$ for the x- and y-position, respectively, starting from the current quadcopter state $x_k$. The attitude and angular-rate references are always set to zero, resulting in the full reference trajectory $x_0^{\text{ref}}, \dots, x_J^{\text{ref}}$ provided to \eqref{eq_smpc_inner_layer_robust}. The cost matrix $Q$
of the MPC is a diagonal matrix with fixed entries, except for the first two diagonal elements being replaced by $q_x$ and $q_y$ depending on the respective command. The commands therefore affect not only the position reference of the MPC, but also its cost matrix $Q$. This allows the MDP to adjust the aggressiveness of \eqref{eq_smpc_inner_layer_robust} online.

To further reduce the computational complexity, we again exploit that the quadcopter dynamics \eqref{eq_linear_model_quadcopter} are invariant to the quadcopter position. Instead of computing $\Tilde{\mathcal{Q}}$ through exhaustive simulation of every state-action pair, we record $200$ trajectories $x^i(t)$ for every command $a\in\mathcal{A}$, $t\in[0,\Delta_t]$, $i\in[200]$, starting at the origin $x^i(0)=0$. During the value iteration \eqref{eq_robust_value_iteration}, for every $c\in\mathcal{C}$, we empirically compute the expectations using the collected trajectories $x^i(t)+c$, $i\in[200]$, for every command $a\in\mathcal{A}$.
Since $\mathcal{E}(t)=\mathcal{H}$ for all $t\geq 0$, the robustification in \eqref{eq_robust_value_iteration} is efficiently implemented by associating every transition to the adjacent partition with the worst cost $\min_{c\in\mathcal{C}_j}\tilde{V}_{k+1}(c)$. 

We first verify the versatility of our approach in scenarios with different safe and target sets $\mathcal{S}$ and $\mathcal{T}$, respectively. We then extend our method to constrained optimal control tasks.

\subsection{Evaluation in Multiple Reach-Avoid Scenarios}
\label{sec_num_scenarios}
We execute our algorithm on multiple reach-avoid scenarios depicted in Fig.~\ref{fig_num_lab} and \ref{fig_num_scenarios}. The resulting reach-avoid probabilities for each scenario are depicted in Table~\ref{tab_scenario_results}. For each scenario the table lists $\Tilde{V}_0({x_0})$, the reach-avoid probability obtained by the robust value iteration \eqref{eq_robust_value_iteration}, and $V_0({x_0})$, the reach-avoid probability achieved when applying the policy $\pi$ obtained from \eqref{eq_robust_value_iteration} to the system. The value of $V_0({x_0})$ is approximately computed through $100$ Monte-Carlo simulations.

\begin{figure}[!t]
  \centering

  \begin{minipage}[b]{0.49\linewidth} 
    \centering
    \includegraphics[width=\linewidth, trim={2.5cm 0cm 2.5cm 1cm}, clip]{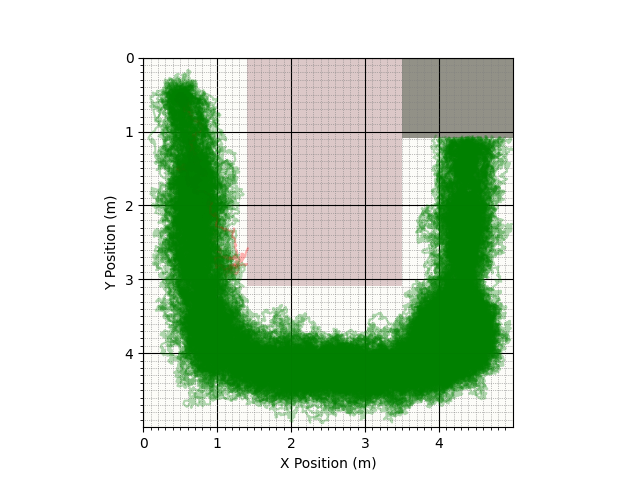}
  \end{minipage}
  \hfill
  \begin{minipage}[b]{0.49\linewidth} 
    \centering
    \includegraphics[width=\linewidth, trim={2.5cm 0cm 2.5cm 1cm}, clip]{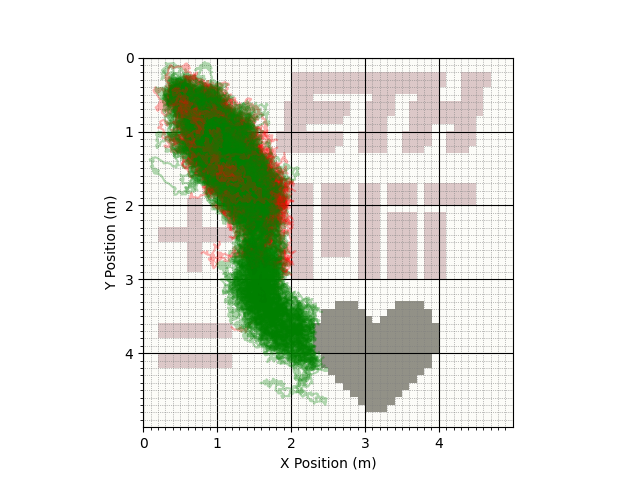}
  \end{minipage}

  \vskip\baselineskip

  \begin{minipage}[b]{0.49\linewidth} 
    \centering
    \includegraphics[width=\linewidth, trim={2.5cm 0cm 2.5cm 1cm}, clip]{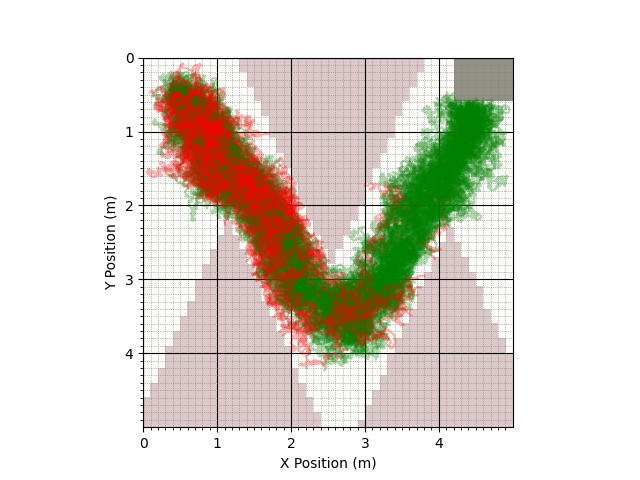}
  \end{minipage}
  \hfill
  \begin{minipage}[b]{0.49\linewidth} 
    \centering
    \includegraphics[width=\linewidth, trim={2.5cm 0cm 2.5cm 1cm}, clip]{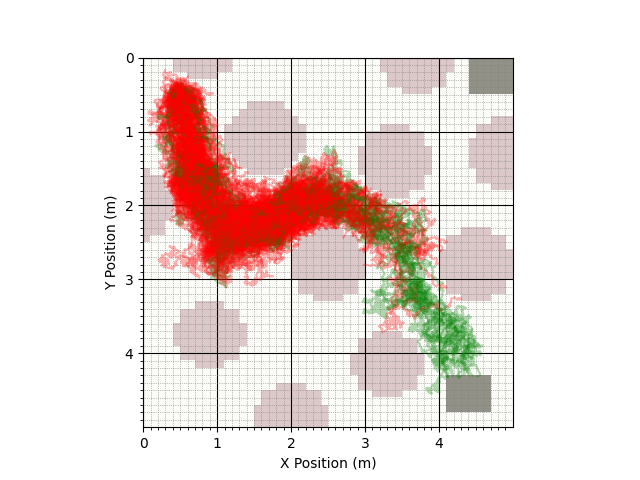}
  \end{minipage}

  \caption{X-Y-trajectories of a quadcopter under our proposed control architecture in different scenarios. The control architecture maximizes the reach-avoid probability. Top-left: simple, Top-right: eth+mit, Bottom-Left: zigzag, Bottom-Right: balls, and Fig.~\ref{fig_num_lab}: labyrinth. The unsafe set $\mathcal{S}^c$ is red, the target set $\mathcal{T}$ gray. Green trajectories satisfy the reach-avoid objective, red trajectories do not. The initial x- and y-position are always chosen as $(0.5,0.5)$.}
  \label{fig_num_scenarios}
\end{figure}

Generally, the longer the path to the target and the smaller safe passages in between unsafe sets, the lower the reach-avoid probability ($\SI{91}{\percent}$ for the \emph{simple}-environment, and $\SI{10}{\percent}$ for the \emph{balls}-environment). As indicated by Theorem \ref{thm_value_function_underapproximation}, the empirical reach-avoid probability is generally higher than the reach-avoid probability computed by \eqref{eq_robust_value_iteration}, $\Tilde{V}_0({x}_0)\leq V_0({x}_0)$. This difference may be significant when the trajectory passes through long and small paths in between unsafe sets, such as for the \emph{zigzag} environment ($\Tilde{V}_0({x}_0)=0.003<V_0({x}_0)=0.44$), and can be reduced through finer gridding. The \emph{eth+mit} scenario ($0.72=\Tilde{V}_0({x}_0)\nleq V_0({x}_0)=0.59$) and \emph{simple} scenario ($\Tilde{V}_0({x}_0)=0.995\nleq V_0({x}_0)=0.91$) are associated higher reach-avoid probabilities in \eqref{eq_robust_value_iteration} than empirically achieved on the real system, which may seem contradictory to Theorem~\ref{thm_value_function_underapproximation}. However, note that the transitions of every state-action pair are estimated using only $200$ samples and the empirical reach-avoid probability is computed using only $100$ Monte-Carlo simulations, causing statistical errors that are not addressed in this work. 

A plot of the references commanded by the policy $\pi$ in the \emph{simple} environment and the corresponding system trajectory are depicted in Fig.~\ref{fig_num_traj_simple}. Notably, the trajectory remains for $10$ time-steps in the bottom right corner. The policy, aware of the $40$ time-steps long horizon, waits for the noise to push the state in a favorable position before attempting to reach the target set via a small passage.

Generating $200$ trajectories for every action took approximately $10$ hours of computation, while running the DP recursion took approximately $72$ minutes for each scenario. All computations have been carried out on a i7-9750H CPU with 16 GB RAM.

\begin{table}[!tb]
\begin{tabularx}{\linewidth}{@{}XXX@{}}\toprule
\textbf{Scenario} & \textbf{$\Tilde{V}_0({x}_0)$} & \textbf{$V_0({x}_0)$}  \\\midrule
        Simple & 0.995 & 0.91 \\
        Eth+mit & 0.720 & 0.59 \\
        Zigzag & 0.003 & 0.44 \\
        Balls & 1.78e-06 & 0.10 \\
        Labyrinth & 0.003 & 0.40 \\\bottomrule\\
\end{tabularx}
\caption{Reach-avoid probability $\Tilde{V}_0({x}_0)$ computed using the robust value iteration \eqref{eq_robust_value_iteration} and the empirical reach-avoid probability $V_0({x}_0)$ achieved by $\pi$ in the quadcopter simulation over $100$ trials for all scenarios in Section \ref{sec_num_scenarios}.}
    \label{tab_scenario_results}
\end{table}

\begin{figure}
    \centering
    \includegraphics[width=.8\columnwidth, trim={2.5cm 0cm 2.5cm 1cm}, clip]{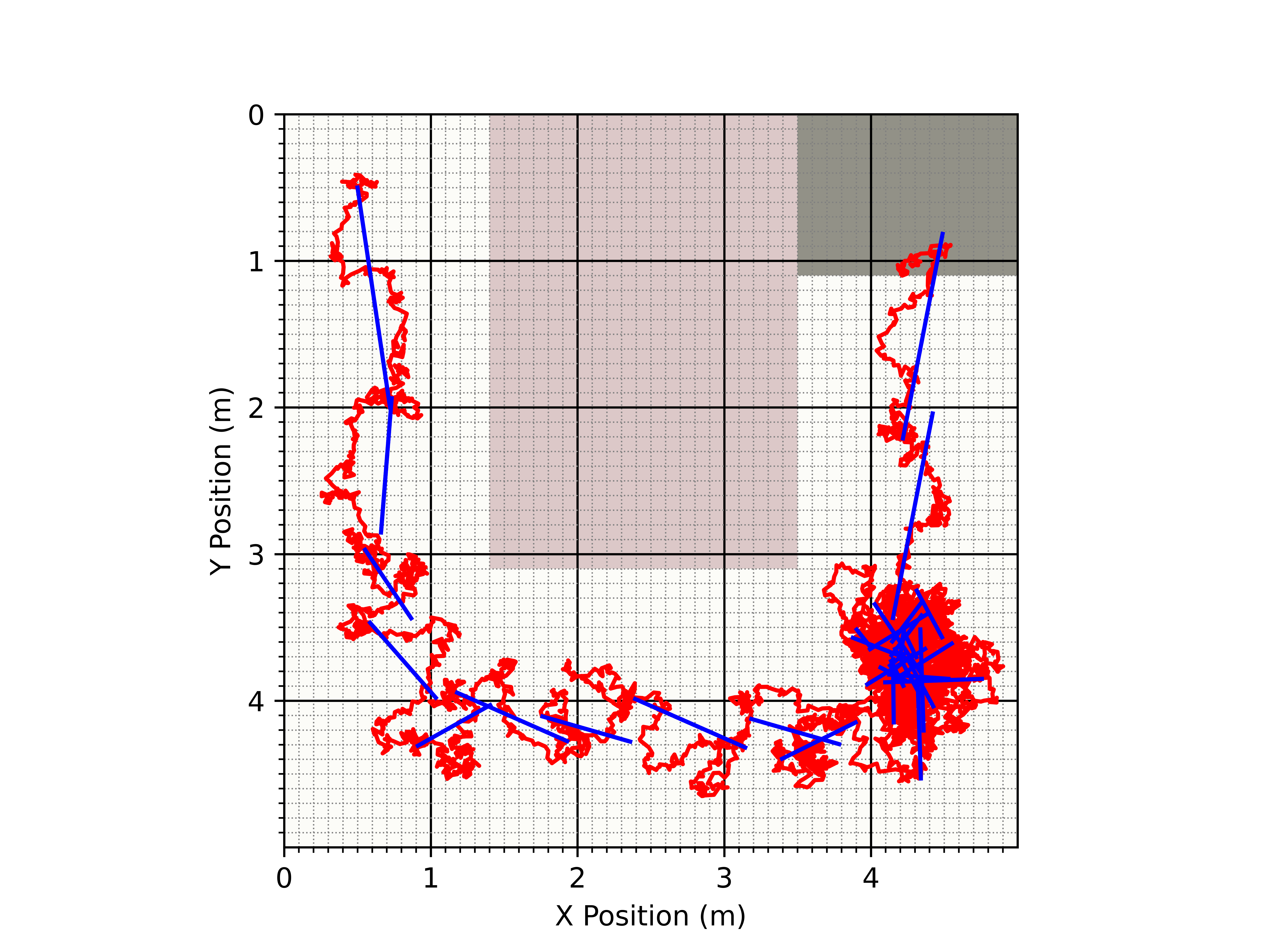}
    \caption{Commands (blue) and quadcopter trajectory (red) of a simulation run in the \emph{simple} environment.}
    \label{fig_num_traj_simple}
\end{figure}

\subsection{Constrained Optimal Control}
\label{sec_num_jcc}
Maximizing the reach-avoid probability can lead to conservative control behaviours. As an alternative, we demonstrate how our architecture can be applied to joint chance constrained optimal control problems to trade off costs against the reach-avoid probability. The goal is to find a controller $\pi^{\star}:\mathcal{X}\rightarrow\mathcal{U}$ solving
\begin{subequations}
\label{eq_num_jcc_problem}
        \label{eq_problem_formulation_jcc}%
    \begin{align}
       &&{\max}_{\pi} & \ \mathbb{E}\left[g_T(x(T)) + \int_{t=0}^{T}\hspace{-.5em}g(x(t),u(t))dt\right] \\
        && \text{s.t.} & \ \dot x(t)=A_cx(t)+B_cu(t)+E_cw(t) \\
        &&& \ \mathbb{P}(\exists t\!\in\![0,\!T]\ |\ x([0,t))\!\in\!\mathcal{G}, x(t)\!\in\!\mathcal{T})\!\geq\!\alpha \\
        &&&\ x(t)\in\mathcal{G}, \ u(t)\in\mathcal{U}, \ u(t)=\pi(x(t)) \\
        &&&\ \forall t\in[0,T].
    \end{align}
\end{subequations} 
Choosing $g(x,\cdot)\!=\!-\frac{\sqrt{(p_x-2.5)^2+(p_y-2.5)^2}}{\sqrt{2}\sqrt{2.5^2}}$, $g_T(\cdot)\!=\!0$ within $\mathcal{S}\setminus\mathcal{T}$, and a cost of zero within $\mathcal{S}^c$ and $\mathcal{T}$, we incentivize the policy to keep the trajectory close to the center $(2.5,2.5)$. We define $g_N\!=\!g_T$, $g_k({x}_k,\!a_k)\!=\!\mathbb{E}\!\left[\int_{t=k\Delta_t}^{(k+1)\Delta_t}g(x(t),\!u(t))dt\middle|x_k,\!a_k\right]$ and $\ell_k$, $\ell_N$ as in Section~\ref{sec_arch_outer}, and formulate the constrained MDP objective
\begin{subequations}%
    \begin{align}
        &\max_{\pi\in\Pi} &&\mathbb{E}\left[g_N({x}_N) + \sum_{k=0}^{N-1}g_k({x}_k,a_k)\right] \\
        & \text{s.t.} && \mathbb{E}[\ell_N({x}_N)]\geq \alpha.
    \end{align}%
    \label{eq_constrained_MDP}%
\end{subequations}%
Given the existence of a strictly feasible policy for which $\mathbb{P}(\exists t\!\in\![0,\!T]\ |\ x([0,t))\!\in\!\mathcal{G}, x(t)\!\in\!\mathcal{T})\!>\!\alpha$, \eqref{eq_constrained_MDP} is optimally solved through its Lagrange dual \cite{schmid2023Jccoc}
{\fontsize{9.5pt}{11.4pt}\selectfont
\begin{align}
\label{eq_num_jcc_lagrange}
\min_{\kappa\geq 0}\max_{\pi\in\Pi} \mathbb{E}\!\Big[g_N({x}_N)\!+\!\kappa(\ell_N({x}_N)\!-\!\alpha)\!+\!\textstyle\sum_{k=0}^{N\!-\!1}g_k({x}_k,a_k) \Big]\!.
\end{align}
}
The parameter $\kappa$ balances the expected costs against the reach-avoid probability. Instead of solving for a maximizing $\kappa$, we simply probe a range of $\kappa\in [4,6,8,10]$ and solve the inner maximization using DP, resulting in the Pareto front in Fig.~\ref{fig_num_jcc_pareto_front}. The trajectories generated under the policy associated to $\kappa=6$ in the \emph{simple} environment are depicted in Fig.~\ref{fig_num_jcc_trajectories}. Indeed, compared to the maximum reach-avoid policy in Fig.~\ref{fig_num_scenarios}, the trajectories remain closer to the center and a higher objective of $-1367$ is achieved compared to $-3509$ by the maximum reach-avoid policy. Yet, this comes at an increased empirical risk of failing the reach-avoid specification of $V_0({x}_0)=0.86$ ($\Tilde{V}_0({x}_0)=0.70$) compared to $V_0({x}_0)=0.91$ ($\Tilde{V}_0({x}_0)=0.995$). While Theorem \ref{thm_value_function_underapproximation} may not apply to Problem \eqref{eq_constrained_MDP}, robustification schemes for constrained interval MDPs are available in the literature \cite{hahn2019interval}. 

\begin{figure}[!tb]
    \centering
    \includegraphics[width = \columnwidth, trim={.5cm .25cm 0.75cm 0cm}, clip]{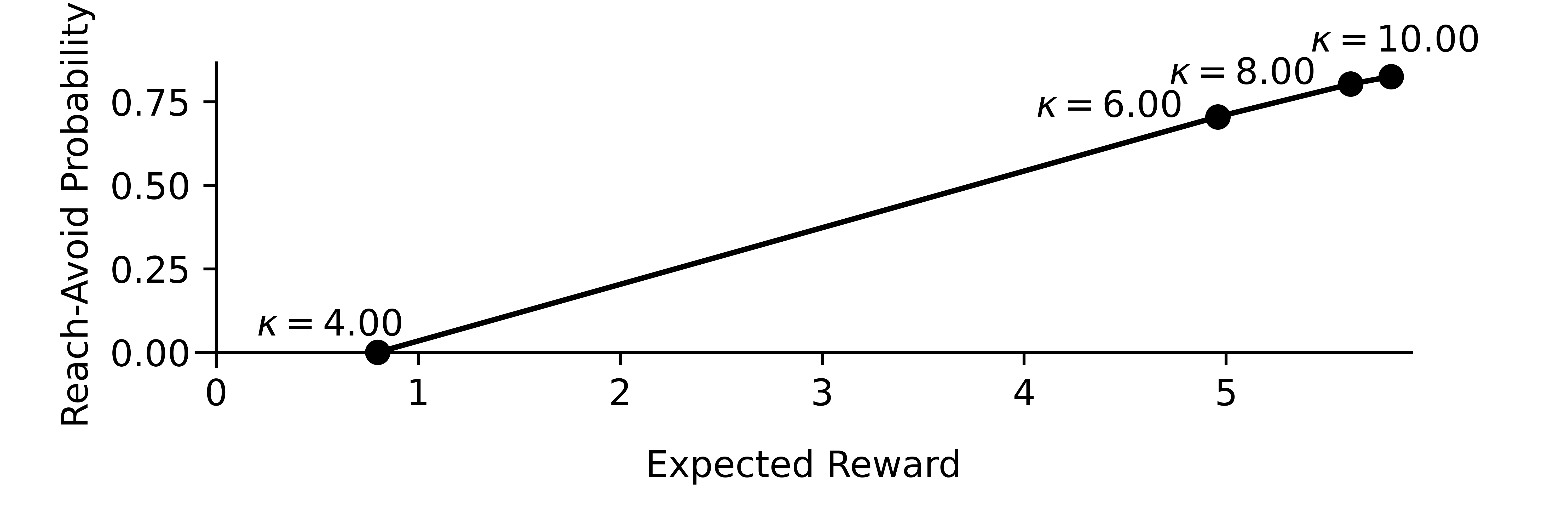}
    \caption{Pareto front of the joint chance constrained optimal control problem \eqref{eq_problem_formulation_jcc}, approached via its Lagrange dual \eqref{eq_num_jcc_lagrange} for different $\kappa\in[4,6,8,10]$.}
    \label{fig_num_jcc_pareto_front}
\end{figure}

\begin{figure}[tbh!]
    \centering
    \includegraphics[width=.8\columnwidth, trim={2.5cm 0cm 2.5cm 1cm}, clip]{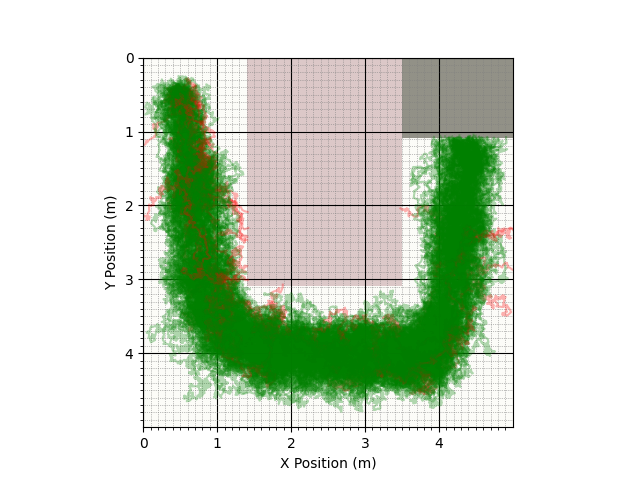}
    \caption{X-Y-trajectories generated by our control architecture for Problem \eqref{eq_num_jcc_lagrange} with $\kappa=6$.}
    \label{fig_num_jcc_trajectories}
\end{figure}

\section{Conclusion}
\label{sec_conclusion}
We analyzed control architectures, consisting of a Model Predictive Controller and Dynamic Programming based reference generator, to maximize reach-avoid probabilities for linear stochastic systems. We derived theoretical guarantees on the achieved reach-avoid probability and numerically validated the approach on a simulated quadcopter. 

Future work may include methodological aspects, such as extending the framework to nonlinear stochastic systems, as well as theoretical discussions, such as conditions ensuring Assumption \ref{ass_existence_of_kernel}.

\bibliographystyle{ieeetr}        
\bibliography{main}    

@article{pitchford2007uncertainty,
  title={Uncertainty and sustainability in fisheries and the benefit of marine protected areas},
  author={Pitchford, Jonathan W and Codling, Edward A and Psarra, Despina},
  journal={Ecol. Model.},
  volume={207},
  number={2-4},
  year={2007},
  publisher={Elsevier}
}

@inproceedings{schmid2022real,
  title={{A real-time GP based MPC for quadcopters with unknown disturbances}},
  author={Schmid, Niklas and Gruner, Jonas and Abbas, Hossam S and Rostalski, Philipp},
  booktitle={ACC},
  year={2022},
  organization={IEEE}
}

@inproceedings{bansal2021deepreach,
  title={Deepreach: A deep learning approach to high-dimensional reachability},
  author={Bansal, Somil and Tomlin, Claire J},
  booktitle={ICRA},
  pages={1817--1824},
  year={2021},
  organization={IEEE}
}

@article{bavdekar,
  title={Stochastic nonlinear model predictive control with joint chance constraints},
  author={Bavdekar, Vinay A and Mesbah, Ali},
  journal={IFAC-PapersOnLine},
  volume={49},
  number={18},
  year={2016},
  publisher={Elsevier}
}

@article{cosner2023robust,
  title={Robust safety under stochastic uncertainty with discrete-time control barrier functions},
  author={Cosner, Ryan K and Culbertson, Preston and Taylor, Andrew J and Ames, Aaron D},
  journal={arXiv preprint arXiv:2302.07469},
  year={2023}
}

@article{lindemann2023safe,
  title={Safe planning in dynamic environments using conformal prediction},
  author={Lindemann, Lars and Cleaveland, Matthew and Shim, Gihyun and Pappas, George J},
  journal={RA-L},
  volume={8},
  number={8},
  pages={5116--5123},
  year={2023},
  publisher={IEEE}
}

@article{schildbach2014scenario,
  title={The scenario approach for stochastic model predictive control with bounds on closed-loop constraint violations},
  author={Schildbach, Georg and Fagiano, Lorenzo and Frei, Christoph and Morari, Manfred},
  journal={Automatica},
  volume={50},
  number={12},
  pages={3009--3018},
  year={2014},
  publisher={Elsevier}
}

@article{thorpe2019model,
  title={Model-free stochastic reachability using kernel distribution embeddings},
  author={Thorpe, Adam J and Oishi, Meeko MK},
  journal={L-CSS},
  volume={4},
  number={2},
  pages={512--517},
  year={2019},
  publisher={IEEE}
}

@article{miller2024unsafe,
  title={Unsafe probabilities and risk contours for stochastic processes using convex optimization},
  author={Miller, Jared and Tacchi, Matteo and Henrion, Didier and Sznaier, Mario},
  journal={arXiv preprint arXiv:2401.00815},
  year={2024}
}

@article{abate2010approximate,
  title={Approximate model checking of stochastic hybrid systems},
  author={Abate, Alessandro and Katoen, Joost-Pieter and Lygeros, John and Prandini, Maria},
  journal={Eur. J. Control},
  volume={16},
  number={6},
  pages={624--641},
  year={2010},
  publisher={Elsevier}
}

@article{abate2008probabilistic,
  title={Probabilistic reachability and safety for controlled discrete time stochastic hybrid systems},
  author={Abate, Alessandro and Prandini, Maria and Lygeros, John and Sastry, Shankar},
  journal={Automatica},
  volume={44},
  number={11},
  year={2008},
  publisher={Elsevier}
}

@article{schmid2022probabilistic,
  title={Probabilistic reachability and invariance computation of stochastic systems using linear programming},
  author={Schmid, Niklas and Lygeros, John},
  journal={IFAC-PapersOnLine},
  volume={56},
  number={2},
  year={2023},
  publisher={Elsevier}
}

@article{schmid2023Jccoc,
  title={Computing Optimal Joint Chance Constrained Control Policies},
  author={Schmid, Niklas and Fochesato, Marta and Li, Sarah H.~Q. and Sutter, Tobias and Lygeros, John},
  journal={arXiv preprint arXiv:2312.10495},
  year={2023}
}

@article{hahn2019interval,
  title={Interval {Markov} decision processes with multiple objectives: from robust strategies to Pareto curves},
  author={Hahn, Ernst Moritz and Hashemi, Vahid and Hermanns, Holger and Lahijanian, Morteza and Turrini, Andrea},
  journal={TOMACS},
  volume={29},
  number={4},
  year={2019},
  publisher={ACM New York, NY, USA}
}

@inproceedings{prandini2008application,
  title={Application of reachability analysis for stochastic hybrid systems to aircraft conflict prediction},
  author={Prandini, Maria and Hu, Jianghai},
  booktitle={CDC},
  pages={4036--4041},
  year={2008},
  organization={IEEE}
}

@book{borrelli2017predictive,
  title={Predictive control for linear and hybrid systems},
  author={Borrelli, Francesco and Bemporad, Alberto and Morari, Manfred},
  year={2017},
  publisher={Cambridge University Press}
}

@article{ames2016control,
  title={Control barrier function based quadratic programs for safety critical systems},
  author={Ames, Aaron D and Xu, Xiangru and Grizzle, Jessy W and Tabuada, Paulo},
  journal={TAC},
  volume={62},
  number={8},
  pages={3861--3876},
  year={2016},
  publisher={IEEE}
}

@article{karaman2011sampling,
  title={Sampling-based algorithms for optimal motion planning},
  author={Karaman, Sertac and Frazzoli, Emilio},
  journal={IJRR},
  volume={30},
  number={7},
  pages={846--894},
  year={2011},
  publisher={Sage Publications Sage UK: London, England}
}

@article{matni2024towards,
  title={Towards a theory of control architecture: A quantitative framework for layered multi-rate control},
  author={Matni, Nikolai and Ames, Aaron D and Doyle, John C},
  journal={arXiv preprint arXiv:2401.15185},
  year={2024}
}

@article{akella2024risk,
  title={Risk-aware robotics: Tail risk measures in planning, control, and verification},
  author={Akella, Prithvi and Dixit, Anushri and Ahmadi, Mohamadreza and Lindemann, Lars and Chapman, Margaret P and Pappas, George J and Ames, Aaron D and Burdick, Joel W},
  journal={arXiv preprint arXiv:2403.18972},
  year={2024}
}

@article{lygeros1999controllers,
  title={Controllers for reachability specifications for hybrid systems},
  author={Lygeros, John and Tomlin, Claire and Sastry, Shankar},
  journal={Automatica},
  volume={35},
  number={3},
  pages={349--370},
  year={1999},
  publisher={Elsevier}
}

@article{garcia2015comprehensive,
  title={A comprehensive survey on safe reinforcement learning},
  author={Garc{\i}a, Javier and Fern{\'a}ndez, Fernando},
  journal={JMLR},
  volume={16},
  number={1},
  pages={1437--1480},
  year={2015}
}

@article{stamouli2025layered,
  title={Layered Multirate Control of Constrained Linear Systems},
  author={Stamouli, Charis and Tsiamis, Anastasios and Morari, Manfred and Pappas, George J},
  journal={arXiv preprint arXiv:2504.10461},
  year={2025}
}

@article{dong2023review,
  title={A review of mobile robot motion planning methods: from classical motion planning workflows to reinforcement learning-based architectures},
  author={Dong, Lu and He, Zichen and Song, Chunwei and Sun, Changyin},
  journal={J Syst Eng Electron},
  volume={34},
  number={2},
  pages={439--459},
  year={2023},
  publisher={BIAI}
}

@inproceedings{fan2020fast,
  title={Fast and guaranteed safe controller synthesis for nonlinear vehicle models},
  author={Fan, Chuchu and Miller, Kristina and Mitra, Sayan},
  booktitle={CAV},
  pages={629--652},
  year={2020},
  organization={Springer}
}

@article{feher2020hierarchical,
  title={Hierarchical evasive path planning using reinforcement learning and model predictive control},
  author={Feh{\'e}r, {\'A}rp{\'a}d and Aradi, Szil{\'a}rd and B{\'e}csi, Tam{\'a}s},
  journal={IEEE Access},
  volume={8},
  pages={187470--187482},
  year={2020},
  publisher={IEEE}
}

@article{stulp2012reinforcement,
  title={Reinforcement learning with sequences of motion primitives for robust manipulation},
  author={Stulp, Freek and Theodorou, Evangelos A and Schaal, Stefan},
  journal={T-RO},
  volume={28},
  number={6},
  pages={1360--1370},
  year={2012},
  publisher={IEEE}
}

@inproceedings{ono_jcc_mpc,
  title={Joint chance-constrained model predictive control with probabilistic resolvability},
  author={Ono, Masahiro},
  booktitle={ACC},
  year={2012}
}

@article{benders2025embedded,
  title={Embedded hierarchical MPC for autonomous navigation},
  author={Benders, Dennis and K{\"o}hler, Johannes and Niesten, Thijs and Babu{\v{s}}ka, Robert and Alonso-Mora, Javier and Ferranti, Laura},
  journal={T-RO},
  year={2025},
  publisher={IEEE}
}

@book{rawlings2020model,
  title={Model predictive control: theory, computation, and design},
  author={Rawlings, James Blake and Mayne, David Q and Diehl, Moritz and others},
  volume={2},
  year={2020},
  publisher={Nob Hill Publishing Madison, WI}
}

\begin{appendices}

\section{Proof of Lemma \ref{lem_cont_time_constraint_satisfaction}}
\label{app_proof_lem_cont_time_constraint_satisfaction}
We first establish the following result on the growth of the state over an interval $\delta_t$.
\begin{lemma}                       
    Consider the system $\dot z(t)=Az(t) + d$, $d\in\mathbb{R}^n$, $t\in[0,\delta_t]$, let
    $\psi(z(0),d)=(e^{\|\!A\!\|\delta_t}\!-\!1)\|z(0)\|\!+\!\frac{1}{\|\!A\!\|}(e^{\|\!A\!\|\delta_t}\!-\!1) \|d\|$ if $\|A\|\!>\!0$ and $\psi(z(0),d) = \delta_t \|d\|$ otherwise. Then, if $hz(0)+\|h\|\psi(z(0),d)\leq b$, it follows that $hz(t)\leq b$ for all $t\in[0,\delta_t]$.
    \label{lem_ct_boundedness}
\end{lemma}
\begin{proof}
        Note that $\|e^{At}-I\| = \|\sum_{k=1}^{\infty}\frac{(At)^k}{k!}\|\leq\sum_{k=1}^{\infty}\frac{\|A\|^kt^k}{k!}=e^{\|A\|t}-1$. For any $t\in[0,\delta_t]$ the deviation of the state $z(t)$ from $z(0)$ is bounded by
    \begin{align*}
        &\max_{t\in[0,\delta_t]}\|z(t)-z(0)\|\\
        &= \max_{t\in[0,\delta_t]}\|e^{At}z(0) + \int_{\tau=0}^{t}e^{A\tau}d\tau d - z(0)\|
        \\
        &\leq \max_{t\in[0,\delta_t]}(e^{\|A\|t}\!-\!1)\|z(0)\|\!+\!\int_{\tau=0}^{t}e^{\|A\|\tau}d\tau \|d\| 
        \\&=\begin{cases}
            \begin{aligned}[b]
            \max_{t\in[0,\delta_t]}&(e^{\|A\|t}-1)\|z(0)\| \\
            &+\!\frac{1}{\|A\|}(e^{\|A\|t}\!-\!1)\|d\|
            \end{aligned}
            &\text{if } \|A\|\!>\!0 \\
            \begin{aligned}[b]
            \max_{t\in[0,\delta_t]}& t \|d\|
            \end{aligned}
            &\text{otherwise.}
        \end{cases}
        \\&
        =\begin{cases}
            \begin{aligned}[b]
            &(e^{\|A\|\delta_t}-1)\|z(0)\| \\
            &\quad\!+\!\frac{1}{\|A\|}(e^{\|A\|\delta_t}\!-\!1)\|d\|
            \end{aligned}
            &\text{if } \|A\|\!>\!0 \\
            \begin{aligned}[b]
            &\delta_t \|d\|
            \end{aligned}
            &\text{otherwise.}
        \end{cases}\\
        &\eqqcolon \psi(z(0),d).
    \end{align*}
    Consequently, assuming that $hz(0)+\|h\|\psi(z(0),d) \leq b$, we have
    \begin{align*}
    	hz(t) &\leq hz(0) + \max_{t\in[0,\delta_t]}h(z(t)-z(0)) 
    	\\&\leq hz(0) + \|h\|\max_{t\in[0,\delta_t]}\|z(t)-z(0)\| 
    	\\&\leq  hz(0)+\|h\|\psi(z(0),d)  
    	\\&\leq b,
    \end{align*} 
    for $t\in[0,\delta_t]$.
\end{proof}

Using Lemma \ref{lem_ct_boundedness}, we now prove Lemma \ref{lem_cont_time_constraint_satisfaction}. Assuming zero-order hold of inputs, during a time-step $t\in[j\delta_t,(j+1)\delta_t)$, the deterministic state component of \eqref{eq_ct_lti_system} evolves according to $\dot x^{\text{d}}(t) = A_4x^{\text{d}}(t)+B_2u$. Let $\mathcal{G}^{\text{d}}=\{x^{\text{d}}\in\mathbb{R}^{n_{\text{d}}}\ | \ h_i^{\text{d}}x^{\text{d}}\leq b \ \forall i\in[M_h]\}$ and 
\begin{align*}
    \eta &= \max_{x\in\mathcal{G}^{\text{d}}, u\in\mathcal{U}}\psi(x,B_2u)\\ &=\max_{x\in\mathcal{G}^{\text{d}}, u\in\mathcal{U}}\begin{cases}
            \begin{aligned}[b]
            &(e^{\|A_4\|\delta_t}-1)\|x\| \\
            &\quad\!+\!\frac{1}{\|A_4\|}(e^{\|A_4\|\delta_t}\!-\!1)\|B_2u\|
            \end{aligned}
            &\text{if } \|A_4\|\!>\!0 \\
            \begin{aligned}[b]
            &\delta_t \|B_2u\|
            \end{aligned}
            &\text{otherwise.}
        \end{cases} \\
        &=\begin{cases}
            \begin{aligned}[b]
            &(e^{\|A_4\|\delta_t}-1)\kappa_x \\
            &\quad\!+\!\frac{1}{\|A_4\|}(e^{\|A_4\|\delta_t}\!-\!1)\kappa_u
            \end{aligned}
            &\text{if } \|A_4\|\!>\!0 \\
            \begin{aligned}[b]
            &\delta_t \kappa_u
            \end{aligned}
            &\text{otherwise.}
        \end{cases} 
\end{align*}
If $h_i{x}_{k,j} + \|h_i\|\eta\leq b$ for all $i\in[M_h]$, then $h_i^{\text{d}}{x}^{\text{d}}_{k,j} \leq h_i^{\text{d}}{x}^{\text{d}}_{k,j} + \|h_i\|\eta\leq b$. Since this implies ${x}_{k,j}^{\text{d}}\in\mathcal{G}^{\text{d}}$, we can apply Lemma \ref{lem_ct_boundedness} and state that $h_i^{\text{d}}{x}^{\text{d}}(t)\leq b$ for all $t\in[k\Delta_t+j\delta_t,k\Delta_t+(j+1)\delta_t]$, $i\in[M_h]$. Then, also $h_i{x}(t)\leq b$. Consequently, if $x_{k,j}\in\Tilde{G}=\{x\in\mathcal{X}\ | \ h_i{x}_{k,j}+ \|h_i\|{\eta} \leq b \ \forall i\in[M_h]\}$, also $x(t)\in\mathcal{G}$ for all $t\in[k\Delta_t+j\delta_t,k\Delta_t+(j+1)\delta_t]$. 
\qed

\section{Proof of Proposition \ref{prop_recursive_feasibility_of_MPC}}
\label{app_proof_prop_recursive_feasibility_of_MPC}


Let $v_{j|j},\dots, v_{J-1|j}\in\mathcal{U}$, $z_{j|j},\dots, z_{J-1|j}\in\hat{\mathcal{G}}$, $z_{J|j}\in\hat{\mathcal{G}}_J$ be a feasible solution for \eqref{eq_smpc_inner_layer} at time-step $j\in[J]$, $k\in[N]$, and denote $z_{l|j}^{\text{d}}=\text{proj}_{\text{d}}(z_{l|j})$ for $l=j,\dots,J$. Let, the input $u_{k,j}=v_{j|j}$ be applied to the system and recall that, by constraint \eqref{eq_smpc_inner_layer_initial_state_constraint}, any feasible solution at time-step $j+1$ must yield $z_{j+1|j+1}^{\text{d}}\coloneqq A_4x_{k,j}^{\text{d}}+B_2u_{k,j}=A_4z_{j|j}^{\text{d}}+B_2u_{k,j}=z_{j+1|j}^{\text{d}}$. Then, the input sequence $v_{l|j+1}=v_{l|j}$ for $l=j+1,\dots,J-1$, yields $z_{l|j+1}^{\text{d}}=z_{l|j}^{\text{d}}$ and, by feasibility of $z_{j|j},\dots, z_{J|j}$ at time-step $j$, $h_i^{\text{d}}z_{l|j+1}^{\text{d}} = h_i^{\text{d}}z_{l|j}^{\text{d}} \leq b -\|h_i\|\eta$ for all $i\in[M_h]$. This implies that for all $l=j+1,\dots,J-1$ the input sequence $v_{j+1|j+1},\dots,v_{J|j+1}$ yields  $h_iz_{l|j+1}=\begin{bmatrix}
    \textbf{0} & h_i^{\text{d}}
\end{bmatrix}\begin{bmatrix}
    z_{l|j+1}^{\text{s},\top} & z_{l|j+1}^{\text{d},\top} 
\end{bmatrix}^{\top}\leq b-\|h_i\|\eta$ for all $i\in[M_h]$ and, consequently, $z_{l|j+1}\in\hat{\mathcal{G}}$. By the same arguments $z_{J|j+1}\in\hat{\mathcal{G}}_J$. 

Starting from $x_{k,0}\in\hat{\mathcal{G}}_J$ for some $k\in[N]$, by the definition of a constraint admissible control invariant set, a feasible solution for \eqref{eq_smpc_inner_layer} exists. Then, following the above discussion, \eqref{eq_smpc_inner_layer} is feasible for all $j\in[J]$, and $x_{k,J}\in\hat{\mathcal{G}}_J$. Applying this argument inductively across the inner horizon $j \in [J]$ and outer horizon $k \in [N]$, starting from $x(0)=x_{0,0}\in\hat{\mathcal{G}}_J$, \eqref{eq_smpc_inner_layer} remains feasible for all $j\in[J]$, $k\in[N]$ since $x_{k,J}\in\hat{\mathcal{G}}_J$ for all $k\in[N]$. 

Finally, by Lemma \ref{lem_cont_time_constraint_satisfaction}, if $x_{k,j}\in\hat{\mathcal{G}}$ is satisfied in discrete time for all $k\in[N]$ and $j\in[J]$, $x(t)\in{\mathcal{G}}$ is satisfied in continuous time for all $t\in[0,T]$. Further, $x_k\in\hat{\mathcal{G}}_J$ for all $k\in[N+1]$. \qed

\section{Proof of Lemma \ref{lem_B_bounds_Epsilon}.}
\label{proof_lem_B_bounds_Epsilon}
Consider the trajectories $x(t)$ and $x'(t)$ of system \eqref{eq_ct_lti_system} originating from $x(0)\in\mathcal{H}_i$ and $x'(0)=c_i$, respectively. Assume both trajectories are driven by the same control input $u(\cdot)$ and noise realization $w(\cdot)$. We define the error as $\xi(t)=x(t)-x'(t)$ for all $t\geq 0$. Since the input and noise terms are identical, they cancel in the difference. This yields the autonomous error system $\xi(t)= e^{A_ct}\xi(0)$ with $\xi(0)\in\mathcal{H}$. Then, the maximum error over the interval $[0,\Delta_t]$ is bounded by
\begin{subequations}
        \begin{align}
            \max_{t\in[0,\Delta_t]}\|\xi(t)\|&=\max_{t\in[0,\Delta_t]}\|e^{A_c t}{\xi}(0)\|\\
            & \leq e^{\|A_c\|\Delta_t}\max_{\xi\in\mathcal{H}}\|\xi\|. 
        \end{align}
\end{subequations} 
Consequently, $\bigcup_{t\in[0,\Delta_t]}\{e^{A_ct}\xi\ | \ \xi\in\mathcal{H}\}\subseteq B_r$ with $r=e^{\|A_c\|\Delta_t}\max_{{\xi}\in\mathcal{H}}\|{\xi}\|$. 
\qed

\section{Proof of Proposition \ref{prop_RMPC_guarantees}}
\label{proof_prop_RMPC_guarantees}
We first establish that the optimization problem \eqref{eq_smpc_inner_layer_robust} is feasible at the initial step $j=0$ for any $x_{k,0} \in \Tilde{\mathcal{G}}_{J} \ominus \mathcal{H}$. The initialization constraint \eqref{eq_smpc_inner_layer_initial_state_constraint_robust} sets $z_0=c_i$, which implies $z_0-x_{k,0}\in\mathcal{H}$. Hence, $z_{0}\in\Tilde{\mathcal{G}}_{J}$, and $z_{0}^{\text{d}}=\text{proj}_{\text{d}}(z_0)\in\Tilde{\mathcal{G}}_J^{\text{d}}$. Note that $\Tilde{\mathcal{G}}_{J}=\mathbb{R}^{n_{\text{s}}}\times\Tilde{\mathcal{G}}_{J}^{\text{d}}$, where $\Tilde{\mathcal{G}}_{J}^{\text{d}}$ is a robust control invariant constraint admissible set for the system 
    $z_{l+1}^{\text{d}} = A_4z_{l}^{\text{d}} + B_2v_{l} + d_{l}^{\text{d}}$ for all $\|d_{l}^{\text{d}}\|_{\infty}\leq r+ \frac{1}{2}\zeta$, where the system is subject to constraints $z_{l}^{\text{d}}\in\tilde{\mathcal{G}}$ and $v_{l}\in\mathcal{U}$ for all $l\in\mathbb{N}$. Consequently, there exists an input sequence $v_{0},\dots, v_{J-1}\in\mathcal{U}$ such that for all $j\in[J]$
\begin{align}%
    &A_4z^{\text{d}}_{j}+B_2v_{j}+d_{j}^{\text{d}}\in\Tilde{\mathcal{G}}^{\text{d}}_J  \ \ \ \qquad \forall \|d_{j}^{\text{d}}\|_{\infty}\leq r+\frac{1}{2}\zeta
    \\ \Leftrightarrow    &Az_{j}+Bv_{j}+d_{j}\in\Tilde{\mathcal{G}}_J\subseteq\Tilde{\mathcal{G}}  \qquad \forall d_{j}\in\mathcal{H}\oplus B_r
    \\ \Leftrightarrow
    &Az_{j}+Bv_{j}\in\Tilde{\mathcal{G}}_J\ominus\mathcal{H}\ominus B_r\subseteq\Tilde{\mathcal{G}}\ominus\mathcal{H}\ominus B_r,
\end{align}
which means problem \eqref{eq_smpc_inner_layer_robust} is feasible at the time-step $j=0$.

Next, we show by induction that feasibility is maintained for all $j \in \{1, \dots, J-1\}$. Consider any $k\in[N]$, $j\in\{1,\dots,J-1\}$. Let $z_{l|j-1}$, $v_{l|j-1}$ denote a feasible solution for \eqref{eq_smpc_inner_layer_robust} at time-step $j-1$ with $l=j-1,\dots,J$ and let $v_{j-1|j-1}$ be applied to the system. Then, 
\begin{subequations}
    \begin{align*}
            z_{j|j} &= x_{k,j} - A^j(c_i-x_{k,0}) \\
        &= Ax_{k,j-1} + Bv_{j-1|j-1} + Ew_{k,j-1} - A^j(c_i-x_{k,0})\\
        &= A(z_{j-1,j-1}+ A^{j-1}(c_i-x_{k,0}))&\\ &\quad + Bv_{j-1|j-1} + Ew_{k,j-1} - A^j(c_i-x_{k,0})\\
        &= Az_{j-1|j-1} + Bv_{j-1|j-1}+ Ew_{k,j-1} \\
        &= z_{j|j-1}+ Ew_{k,j-1}.
    \end{align*}
\end{subequations}
More generally, choosing $v_{l|j}=v_{l|j-1}$ for $l=j,\dots,J$, yields 
\begin{align*}
    z_{l|j} 
    &= z_{l|j-1}+A^{l-j}Ew_{k,j-1}.
\end{align*}
By design, for all $i\in[M_h]$, if
\begin{align*}
&\begin{bmatrix}
    \textbf{0} & h_i^{\text{d}}
\end{bmatrix}\begin{bmatrix}
    z_{l|j}^{\text{s},\top} & z_{l|j}^{\text{d},\top}
\end{bmatrix}^{\top} + \|h_i\|(\eta+r)\leq b,
\intertext{then}
h_iz_{l|j} =
& \begin{bmatrix}
    \textbf{0} & h_i^{\text{d}}
\end{bmatrix}\begin{bmatrix}
    z_{l|j-1}^{\text{s}}\!+\!A_1^{l-j}E_1w_{k,j-1}\\ z_{l|j-1}^{\text{d}}
\end{bmatrix}\!+\!\|h_i\|(\eta\!+\!r)\!\leq\!b,
\end{align*}
which is $z_{l|j}\in\Tilde{\mathcal{G}}$. The same argumentation leads to $z_{J|j}\in\Tilde{\mathcal{G}}_J\ominus\mathcal{H}\ominus B_r$. 
Since this holds for any $j\in\{1,\dots,J-1\}$, starting from $x_{k,0}\in\Tilde{\mathcal{G}}_J\ominus\mathcal{H}$, for which a feasible input sequence exists based on prior discussion, \eqref{eq_smpc_inner_layer_robust} remains feasible with $z_{j|j}\in\tilde{\mathcal{G}}$ for all $j\in[J]$ and $z_{J|J}\in\Tilde{\mathcal{G}}_J\ominus\mathcal{H}\ominus B_r$. By Corollary \ref{lem_if_in_hat_then_in_tilde}, this implies that $x_{k,j}\in\hat{\mathcal{G}}$ for all $j\in[J+1]$, and by definition of $B_r$ it holds that $x_{k,J}\in\hat{\mathcal{G}}_J\ominus\mathcal{H}$.

By induction, starting from $x(0)\in\Tilde{\mathcal{G}}_J\ominus\mathcal{H}$ and setting $x_{k,0}=x_{k-1,J}$, problem \eqref{eq_smpc_inner_layer_robust} is feasible for all $j\in[J]$ and $k\in[N]$.

To show that the same input is generated from any $x_k\in\mathcal{H}_i$, note that the initial state constraint \eqref{eq_smpc_inner_layer_initial_state_constraint_robust},
\begin{align*}
    &x_{k,j} -A^jx_{k,0}+A^jc_i
    \\  &\quad= A^jx_{k,0} + \sum_{l=0}^{j-1}A^{l}Bu_{k,j-l-1} \\& \qquad + \sum_{l=0}^{j-1}A^{l}Ew_{k,j-l-1} -A^jx_{k,0}+A^jc_i
    \\ &\quad =A^jc_i+\sum_{l=0}^{j-1}A^{l}(Bu_{k,j-l-1} +Ew_{k,j-l-1}),
\end{align*}
decouples the optimization from the exact realization of $x_{k,0}$ within $\mathcal{H}_i$ for all $j\in[J]$. \qed

\section{Proof of Theorem \ref{thm_value_function_underapproximation}}
\label{app_proof_thm_value_function_underapproximation}
For all $k\in[N+1]$, let $\overline{V}_k$ be the piecewise constant extension of $\Tilde{V}_k$ to ${\mathcal{X}}$ such that $\overline{V}_k({x}_k)=\Tilde{V}_k({c}_i)$ whenever ${x}_k\in\mathcal{H}_i$, $i\in[M_x]$.

By definition of $\tilde{\mathcal{S}}$ and $\tilde{\mathcal{T}}$, $\mathcal{H}_j\subseteq \Tilde{\mathcal{S}}\setminus\Tilde{\mathcal{T}}\subseteq\tilde{\mathcal{S}}$ for any $j\in \iota$. Then, for all $j\in\iota$ and $c_i\in \Tilde{\mathcal{S}}^c$, $\tilde{\mathcal{Q}}(c_j|c_i,\cdot)={\mathcal{Q}}(\mathcal{H}_j|c_i,\cdot)\leq {\mathcal{Q}}(\Tilde{\mathcal{S}}|c_i,\cdot)=0$ since $\tilde{\mathcal{S}}^c$ is absorbing (by abuse of notation, we here refer to $\mathcal{Q}$ with absorbing sets $\Tilde{\mathcal{S}}$ and $\Tilde{\mathcal{T}}$ in place of $\mathcal{S}$ and $\mathcal{T}$). Using this fact, for all $x_k\in\mathcal{H}_i\cap\Tilde{\mathcal{S}}^c$, $i\in[M_x]$ and $k\in[N]$,
\begin{align*}
    \overline{V}_k(x_k)\!&=\!\Tilde{V}_k(c_i)
    \\&=\!\Tilde{\mathcal{Q}}(\Tilde{\mathcal{T}}| c_i,\pi_k(c_i))
    \\& \quad +\sum_{j\in\iota}\min_{c\in\mathcal{C}_j}\Tilde{V}_{k+1}({c})\Tilde{\mathcal{Q}}(c_j| c_i,\pi_k(c_i))\\&=\!0,
\end{align*}
and $\overline{V}_{k}(x_k)=0\leq V_k(x_k)$. Further, by construction, $\overline{V}_{k}\in[0,1]$. Consequently, for any $x_k\in\mathcal{T}$ and $k\in[N]$ it holds that $\overline{V}_k(x_k)\leq V_k(x_k)= 1$. 

It remains to show that $\overline{V}_k(x_k)\leq V_k(x_k)$ for all $x_k\in\Tilde{\mathcal{S}}\setminus\mathcal{T}$ and $k\in[N]$.  We use the induction hypothesis that for some $k\in[N]$ the relation $V_{k+1}(x_{k+1})\geq \overline{V}_{k+1}(x_{k+1})$ holds for all $x_{k+1}\in\Tilde{\mathcal{S}}\setminus\mathcal{T}$, and show that this yields $V_{k}(x_k)\geq \overline{V}_{k}(x_k)$ for all $x_k\in\Tilde{\mathcal{S}}\setminus\mathcal{T}$. 

Consider any $x_k\in\mathcal{H}_i$ and $x_k'=c_i$ for some $i\in[M_x]$. Let the noise $w(\cdot)$ and input $u(\cdot)$ for system \eqref{eq_ct_lti_system} be fixed and $x_k$ transition to $x_{k+1}$ with absorbing sets $\mathcal{S}$ and $\mathcal{T}$, and $x'_k$ transition to $x'_{k+1}$ with absorbing sets $\Tilde{\mathcal{S}}$ and $\Tilde{\mathcal{T}}$. Note that, unless a trajectory is absorbed by any of these sets, $x_{k+1}-x'_{k+1}\in B_r$ by Lemma \ref{lem_B_bounds_Epsilon}. 
We consider three cases.
\begin{enumerate}
    \item $x_{k+1}\in\mathcal{T}$: Since $\overline{V}_{k+1}(x'_{k+1})\in[0,1]$, if $x_{k+1}\in\mathcal{T}$, then $\overline{V}_{k+1}(x'_{k+1})\leq V_{k+1}(x_{k+1})=1$;
    \item $x_{k+1}\in\mathcal{S}^c$: If $x_{k+1}\in\mathcal{S}^c$, then $x'_{k+1}\in\Tilde{\mathcal{S}}^c$ and thus $\overline{V}_{k+1}(x'_{k+1})=V_{k+1}(x_{k+1})=0$;
    \item $x_{k+1}\in\mathcal{S}\setminus\mathcal{T}$: If $x_{k+1}\in\mathcal{S}\setminus\mathcal{T}$, then $x'_{k+1}\notin \Tilde{\mathcal{T}}$, but either $x'_{k+1}\in\Tilde{\mathcal{S}}^c$, or $x'_{k+1}\in\Tilde{\mathcal{S}}\setminus\Tilde{\mathcal{T}}$ with $x'_{k+1}-x_{k+1}\in B_r$ by Lemma \ref{lem_B_bounds_Epsilon}. If $x'_{k+1}\in\Tilde{\mathcal{S}}^c$, the bounds holds trivially since $0=\overline{V}_{k+1}(x'_{k+1})\leq V_{k+1}(x_{k+1})$; in the case that $x'_{k+1}\in\Tilde{\mathcal{S}}\setminus\Tilde{\mathcal{T}}$, there exists $j\in\iota$ such that $x'_{k+1}\in\mathcal{H}_j$. Then, $\min_{\epsilon\in B_r\oplus\mathcal{H}}\overline{V}_{k+1}(c_j+\epsilon)\leq \min_{\epsilon\in B_r}\overline{V}_{k+1}(x'_{k+1}+\epsilon)\leq \overline{V}_{k+1}(x_{k+1})\leq {V}_{k+1}(x_{k+1})$, where the latter holds by the induction hypothesis for all $x_{k+1}\in\Tilde{\mathcal{S}}\setminus\mathcal{T}$ and by the prior analysis for all $x_{k+1}\in\Tilde{\mathcal{S}}^c$.
\end{enumerate}
Together, 
\begin{subequations}
    \begin{align}
    \mathds{1}_{\mathcal{T}}(&x_{k+1}) + \mathds{1}_{\mathcal{S}\setminus\mathcal{T}}(x_{k+1})V_{k+1}(x_{k+1})
    \\
    &\geq\!\mathds{1}_{\tilde{\mathcal{T}}}(x'_{k+1})\!+\!\mathds{1}_{\tilde{\mathcal{S}}\setminus\tilde{\mathcal{T}}}(x'_{k+1})\!\min_{\epsilon\in B_r}\!\overline{V}_{k+1}(x'_{k+1}\!+\!\epsilon) 
    \\
    &\geq
    \mathds{1}_{\tilde{\mathcal{T}}}(x'_{k+1})\!+\!\sum_{j\in\iota}\!\mathds{1}_{\mathcal{H}_j}(x'_{k+1})\hspace{-.5em}\min_{\epsilon\in B_r\oplus\mathcal{H}}\hspace{-.7em}\overline{V}_{k+1}(c_j\!+\!\epsilon). 
\end{align}
\label{eq_proof_E_IndicatorFcts}%
\end{subequations}%
For the last inequality, recall that by definition of $\Tilde{\mathcal{S}}$, $\Tilde{\mathcal{T}}$ and $\iota$, $\Tilde{\mathcal{S}}\setminus\Tilde{\mathcal{T}}=\bigcup_{j\in\iota}\mathcal{H}_j$. Since \eqref{eq_proof_E_IndicatorFcts} holds for any noise realization and input, it also holds in expectation. Combined with the fact that $\mathcal{Q}$ in \eqref{eq_true_reach_avoid_probabilities} models transitions with absorbing sets $\mathcal{S}$ and $\mathcal{T}$, while $\Tilde{\mathcal{Q}}$ models transitions with absorbing sets $\Tilde{\mathcal{S}}$ and $\Tilde{\mathcal{T}}$,
\begin{align*}
    \Tilde{\mathcal{Q}}(&\Tilde{\mathcal{T}}| c_i,\pi_k(c_i)) +\sum_{j\in\iota}\min_{\epsilon\in B_r\oplus\mathcal{H}}\overline{V}_{k+1}(c_j+\epsilon)\Tilde{\mathcal{Q}}(c_j| c_i,\pi_k(c_i)) 
    \\
    &=\mathbb{P}({x}'_{k+1}\in\tilde{\mathcal{T}}| c_i,\pi_k(c_i)) \\&\quad +\sum_{j\in\iota}\min_{\epsilon\in B_r\oplus\mathcal{H}}\overline{V}_{k+1}(c_j+\epsilon)\mathbb{P}({x}'_{k+1}\in\mathcal{H}_j| c_i,\pi_k(c_i)) 
    \\
    &=\mathbb{E}[\mathds{1}_{\tilde{\mathcal{T}}}(x'_{k+1}) \\&\quad + \sum_{j\in\iota}\mathds{1}_{\mathcal{H}_j}(x'_{k+1})\min_{\epsilon\in B_r\oplus\mathcal{H}}\overline{V}_{k+1}(c_j+\epsilon) | c_i,\pi_k(c_i)]
    \\
    &\leq\mathbb{E}[\mathds{1}_{\mathcal{T}}(x_{k+1})+ \mathds{1}_{\mathcal{S}\setminus\mathcal{T}}(x_{k+1})V_{k+1}(x_{k+1}) | x_k,\pi_k(x_k)]
    \\
    &= \int_{\mathcal{X}}{V}_{k+1}(x_{k+1})\mathcal{Q}(dx_{k+1}| {x}_k,\pi_k({x}_k))
    \\
    &= V_k({x}_k),
\end{align*}
where \eqref{eq_smpc_inner_layer_robust} generates the same input sequence for $\pi_k(x_k)$ and $\pi_k(c_i)$ under the same noise realization based on Proposition \eqref{prop_RMPC_guarantees}.

Now, consider $x_k\in
\mathcal{H}_i\cap\Tilde{\mathcal{S}}\setminus{\mathcal{T}}$ for some $i\in[M_x]$ and note that $\Tilde{\mathcal{S}}\setminus{\mathcal{T}}\subseteq \mathcal{S}\setminus\mathcal{T}$. Then,
\begin{subequations}
    \begin{align}
        V_k({x}_k)\!
        &\geq \Tilde{\mathcal{Q}}(\Tilde{\mathcal{T}}| c_i,\pi_k(c_i))
         \label{eq_proof_of_abstraction_bound_4}\\& \quad +\sum_{j\in\iota}\min_{\epsilon\in B_r\oplus\mathcal{H}}\overline{V}_{k+1}(c_j+\epsilon)\Tilde{\mathcal{Q}}(c_j| c_i,\pi_k(c_i)) \nonumber
        \\
        & = \Tilde{\mathcal{Q}}(\Tilde{\mathcal{T}}| c_i,\pi_k(c_i))
         \label{eq_proof_of_abstraction_bound_5} \\& \quad +\sum_{j\in\iota}\min_{c\in\mathcal{C}\cap \mathcal{H}_j\oplus B_r\oplus \mathcal{H}}\overline{V}_{k+1}({c})\Tilde{\mathcal{Q}}(c_j| c_i,\pi_k(c_i))\nonumber
         \\
         &=\Tilde{V}_k(c_i)\label{eq_proof_of_abstraction_bound_6}
         \\
         &=\overline{V}_k(c_i),\label{eq_proof_of_abstraction_bound_7}
    \end{align}%
    \label{eq_proof_of_abstraction_bound}%
\end{subequations}%
where \eqref{eq_proof_of_abstraction_bound_4} is based on the prior analysis. Relation \eqref{eq_proof_of_abstraction_bound_5} follows since any point $c_j+\epsilon\in\mathcal{H}_l$ with $\epsilon\in B_r\oplus\mathcal{H}$ yields $c_l\in \mathcal{H}_j\oplus B_r\oplus\mathcal{H}$, $l\in[M_x]$, with $\overline{V}_{k+1}(c_l)=\overline{V}_{k+1}(c_j+\epsilon)$. By symmetry of $B_r$ and the hypercubes, if $c_l\notin \mathcal{H}_j\oplus B_r\oplus\mathcal{H}$, $l\in[M_x]$, then there exists no $\epsilon\in B_r\oplus\mathcal{H}$ with $c_j+\epsilon\in\mathcal{H}_l$, leading to \eqref{eq_proof_of_abstraction_bound_5}. Finally, \eqref{eq_proof_of_abstraction_bound_6} and \eqref{eq_proof_of_abstraction_bound_7} follow by the definition of $\Tilde{V}_k$ in \eqref{eq_robust_value_iteration} and $\overline{V}_k$ as its piecewise constant extension to $\mathcal{X}$. In summary, $V_k({x}_k)\geq \overline{V}_k(x_k)$ for all $x_k\in\mathcal{X}$.


The induction hypothesis is satisfied at $k=N$, where any ${x}_N\in {\mathcal{T}}$ yields $\overline{V}_N({x}_N) \leq V_N({x}_N)=1$ and any ${x}_N\in {\mathcal{T}^c}\subseteq \Tilde{\mathcal{T}}^c$ yields $\overline{V}_N({x}_N) = V_N({x}_N)=0$. It follows that $\overline{V}_k(\cdot)\geq V_k(\cdot)$ for all $k\in[N+1]$. Since ${c}_0,{x}_0\in\mathcal{H}_0$, $\Tilde{V}_0({c}_0)= \overline{V}_0({x}_0)\leq V_0({x}_0)$ yields the claim.\qed


\end{appendices}
\end{document}